\documentclass[reqno]{amsart}

\usepackage[utf8]{inputenc}
\usepackage[T1]{fontenc}
\usepackage[english]{babel}

\usepackage{amssymb}
\usepackage{calrsfs}
\usepackage{array}
\usepackage{bbm}
\usepackage{stmaryrd}

\usepackage{mathtools}
\usepackage[svgnames]{xcolor}
\usepackage{hyperref}
\usepackage{wrapfig}
\usepackage{tikz}
\usetikzlibrary{arrows,calc}
\definecolor{light-gray1}{gray}{0.90}
\definecolor{light-gray2}{gray}{0.80}

\usepackage{hyperref}

\newcommand{\m}[1]{%
\def\arg{#1}%
\def\one{1}%
\ifx\arg\one%
\mathbbm{#1}%
\else%
\mathbb{\expandafter{#1}}
\fi%
}

\newcommand{\q}[1]{\mathcal{#1}}

\newcommand{\mr}[1]{\mathrm{#1}}

\newcommand{\ds}{\displaystyle}

\newcommand{\e}{\varepsilon}

\newcommand{\tend}{\longrightarrow}

\newcommand{\w}{\omega}

\def\ro{\rho}

\newcommand{\EQ}[1]{\begin{equation}\begin{split} #1 \end{split}\end{equation}}

\def\R{\mathbb{R}}
\def\eps{\varepsilon}
\def\wt{\widetilde}
\def\ol{\overline}
\renewcommand\Re{\mathrm{Re}}
\def\I{\infty}
\def\f{\frac}
\def\nn{\nonumber}
\def\ro{\rho}

\def\calS{\mathcal{S}}
\def\lan{\langle}
\def\ran{\rangle}
\def\wh{\widehat}

\def\calF{\mathcal{F}}
\def\sign{\mathrm{sign}}
\def\Z{\mathbb{Z}}
\newcommand{\weak}{\rightharpoonup}
\def\cL{\mathcal{L}}
\def\p{\partial}

\DeclareMathOperator{\Supp}{\mathrm{Supp}}

\theoremstyle{plain}
\newtheorem{thm}{Theorem}%[section]
\newtheorem*{thm*}{Theorem}

\newtheorem{prop}[thm]{Proposition}
\newtheorem{cor}[thm]{Corollary}
\newtheorem{lem}[thm]{Lemma}

\theoremstyle{definition}
\newtheorem{defi}[thm]{Definition}

\theoremstyle{remark}

\makeatletter
\def\blfootnote{\xdef\@thefnmark{}\@footnotetext}
\makeatother

\title{Energy partition for the linear radial wave equation}

\date{}%\nodate
\author{Rapha\"{e}l C\^{o}te \and Carlos E. Kenig \and Wilhelm Schlag}

\allowdisplaybreaks



\thanks{Support of the National Science Foundation DMS-0968472 for the second author, and  DMS-0617854, DMS-1160817 for the third author is gratefully acknowledged. 
This first author wishes to thank the University of Chicago for its hospitality during the academic year 2011-12, and acknowledges support from the European Research Council through the project BLOWDISOL. The authors thank Andrew Lawrie for comments on a preliminary version of this paper.}

\begin{document}

%The abstract of your paper
\begin{abstract}
We consider the radial free wave equation in all dimensions and derive asymptotic formulas for the space partition of the energy, as time goes to infinity.
We show that the exterior energy estimate, which Duyckaerts, Merle and the second author obtained in odd dimensions \cite{DKM11, DKM12}, fails in even dimensions.
%We show that the exterior energy estimate for the radial free wave equation, which Duyckaerts, Merle and the second author obtained in odd dimensions, fails in even dimensions.
Positive results for restricted classes of data are obtained.  
\end{abstract}

\maketitle

\section{Introduction}
\label{intro}

In this paper we consider solutions to the wave equation
\begin{equation}\label{eq:lw}
\Box u=0,\quad u(0)=f,\; u_t(0)=g
\end{equation}
where $(f,g)\in(\dot H^1\times L^2)(\R^d)$ are \emph{radial}. Denote by $u(t) = S(t)(f,g)$ the solution to this wave equation \eqref{eq:lw} with initial data $(f,g)$ at time 0.

The origin of our work lies in the exterior energy estimates
obtained by Duyckaerts, the second author, and Merle~\cite{DKM11}, \cite{DKM12} which state that for $d\ge3$ and odd, one
has either one of the following estimates (even in the nonradial setting):
\begin{equation}
\label{DKMext}
\begin{aligned}
\forall\; t\ge0, \quad \int_{|x|\ge t} |\nabla_{t,x} S(t)(f,g)(x)|^2\, dx &\ge \frac12 \int_{\R^d} (|\nabla f(x)|^2 + |g(x)|^2)\, dx \\
\forall\; t\le0, \quad
\int_{|x|\ge -t} |\nabla_{t,x} S(t)(f,g)(x)|^2 \, dx &\ge \frac12 \int_{\R^d} (|\nabla f(x)|^2 + |g(x)|^2)\, dx, \end{aligned}\end{equation}
where 
\[ |\nabla_{t,x} S(t)(f,g)|^2 = |\nabla u(t)|^2 + |\partial_t u(t)|^2 \]
 is the linear energy density (see \cite[Proposition~2.7]{DKM12}).  No result of this type was established there for even dimensions,
and the method of proof used in odd dimensions does not apply in even dimensions.

\medskip

In this paper we show that  \eqref{DKMext} fails in even dimensions. To be specific, there does not exist a positive constant
which can be substituted on the right-hand side for~$\frac12$ and so that the resulting inequality will hold for all $(f,g)$. 
 This will be based on a computation of the asymptotic exterior energy
\EQ{
\nn %\label{AsEn}
\lim_{t\to\pm\infty} \int_{|x|\ge |t|} |\nabla_{t,x} S(t)(f,g)(x)|^2\, dx 
}
Note that the exterior energy is decreasing in $|t|$, whence \eqref{DKMext} reduces to the computation of these limits. 
Since the propagator $S(t)$ is difficult to work with on the ``physical side'', we employ the Fourier transform in this computation.
To state our asymptotic result, we introduce the Hankel  transform $H$ and the Hilbert transform $\q H$ on the half-line $(0,\infty)$: 
\begin{equation} \nn %\label{HankelHilbert}
(H\varphi)(\rho):=\int_0^\infty \frac{\varphi(\sigma)}{\rho+\sigma}\, d\sigma, \quad \text{and} \quad (\q H\varphi)(\rho):=\int_0^\infty \frac{\varphi(\sigma)}{\rho-\sigma}\, d\sigma
\end{equation}
where the second integral is to be taken in the principal value sense. 
Both these operators are bounded and self-adjoint (anti-selfadjoint, respectively) on $L^2((0,\infty),d\rho)$, with norm $\pi$. Furthermore, $H$ is  
a positive operator since it is of the form $H=\cL^2$ where $\cL=\cL^*$ is the Laplace transform, see for example Lax~\cite{LaxFA02} for details. This positivity is important for our purposes. In even dimensions, we find the following expression for the asymptotic exterior energy in terms of~$H$ and $\q H$. 
In the next two theorems, we use the notation
\[
\langle f,g\rangle:= \int_0^\infty f(x) \ol{g(x)}\, dx
\]
for two functions $f,g$ on the half-line $(0,\infty)$. 

\begin{thm} \label{th1}
Let $d$ be even, $(f,g)\in \dot H^1\times L^2(\R^d)$ be radial as above,
and denote by $\hat{f},\hat{g}$ their Fourier transforms in~$\R^d$. Then for some constant $C(d)>0$ one has for the solution $u$ of~\eqref{eq:lw}  
\begin{multline} \label{eq:ext energy}
\lim_{t\to\pm\infty} C(d)\int_{|x|\ge |t|}  | \nabla_{t,x} S(t)(f,g)|^2\,dx  = 
\frac{\pi}{2} \int  (\ro^2|\hat{f}(\rho)|^2+|\hat{g}(\rho)|^2)\ro^{d-1}\, d\ro \\
 +\frac{(-1)^{\frac{d}{2}} }{2} \left( \langle H(\rho^{\frac{d+1}{2}} \hat f), \rho^{\frac{d+1}{2}} \hat f \rangle - \langle H( \rho^{\frac{d-1}{2}} \hat g), \rho^{\frac{d-1}{2}} \hat g \rangle \right) \pm \Re \langle \rho^{\frac{d+1}{2}} \hat f, \q H( \rho^{\frac{d-1}{2}} \hat g) \rangle. 
\end{multline}
\end{thm}

The constant $C(d)$ is explicit, see below. 
This immediately implies that for $d \equiv 2 \mod 4$, there can be {\em no exterior energy estimate} for
the initial value problem with data~$(f,0)$, whereas there is such an estimate for data of the form~$(0,g)$. Indeed, we infer from~\eqref{eq:ext energy} and the positivity
of the Hankel transform  that 
\begin{multline}\nn
\lim_{t\to\pm\infty} C(d)\int_{|x|\ge |t|}  |\nabla_{t,x} S(t) (0,g)(x)|^2\, dx  = \\
\frac{\pi}{2} \int |\hat{g}(\rho)|^2\ro^{d-1}\, d\ro  + 
  \frac{1}{2}  \langle H (\rho^{\frac{d-1}{2}}\hat g), \rho^{\frac{d-1}{2}}\hat g \rangle \\ \ge \frac{\pi}{2} \int |\hat{g}(\rho)|^2\ro^{d-1}\, d\ro = C_1(d) \|g\|_{L^2(\R^d)}^2.
\end{multline}
On the other hand, since $\|\cL\|_{2\to2}=\sqrt{\pi}$, we see that 
\[
\langle H(\rho^{\frac{d+1}{2}} \hat f), \rho^{\frac{d+1}{2}} \hat f \rangle = \big\| \cL ( \rho^{\frac{d+1}{2}} \hat f) \big\|_2^2
\]
can come arbitrarily close to $\pi \| \rho^{\frac{d+1}{2}} \hat f \|_2^2$ whence no positive definite lower bound in~\eqref{eq:ext energy} is possible for data $(f,0)$. 
%On the other hand,  if $f_n$ is such that $$\| H (\rho^{\frac{d+1}{2}} \hat f_n) \|_{L^2((0,\infty))} \tend \pi \| \rho^{\frac{d+1}{2}} \hat f_n \|_{L^2((0,\infty))}$$ as $n\to\infty$,  then with data $(f_n,0)$ 
%\begin{equation} \nn
%\lim_{t\to\pm\infty} C(d)\int_{|x|\ge |t|}  |\nabla_{t,x} S(t)(f_n,0)(x)|^2\, dx   = o( \| \rho^{\frac{d+1}{2}} \hat f_n \|_{L^2((0,\infty))}^2)=o(\|f_n\|_{\dot H^1(\R^d)}^2). 
%\end{equation}
We remark that $\m 1_{a,b}/\sqrt{ \sigma}$ where $b/a \to +\infty$ is an explicit extremizing  family for~$\cL$. 
Symmetrically,  if $d\equiv 0 \mod 4$, then there is an exterior energy estimate for data~$(f,0)$ but not for~$(0,g)$. 

This is in sharp contrast with the asymptotics for odd dimensions:
\begin{thm} \label{th2}
Let $d$ be odd, $(f,g)\in (\dot H^1\times L^2)(\R^d)$ be radial,
and denote by $\hat{f},\hat{g}$ their Fourier transforms in~$\R^d$. Then for some constant $C(d)>0$ one has for the solution $u$ of~\eqref{eq:lw}  
\begin{multline} \label{eq:ext energy odd}
\lim_{t\to\pm\infty} C(d)\int_{|x|\ge |t|}  | \nabla_{t,x} S(t)(f,g)(x)|^2\, dx  = 
\frac{\pi}{2} \int  (\ro^2|\hat{f}(\rho)|^2+|\hat{g}(\rho)|^2)\ro^{d-1}\, d\ro \\
 \pm \left( (-1)^{\frac{d-1}{2}} \Re \langle H(\rho^{\frac{d+1}{2}} \hat f), \rho^{\frac{d-1}{2}} \hat g \rangle  + \Re \langle \rho^{\frac{d+1}{2}} \hat f, \q H ( \rho^{\frac{d-1}{2}}\hat g) \rangle \right). 
\end{multline}
\end{thm}
>From this one  immediately deduces \eqref{DKMext} up to constants.  We prove Theorem~\ref{th1}, \ref{th2} in Section~\ref{basic}. 
The failure of~\eqref{DKMext}  presents a serious obstruction for the extension of the nonlinear machinery developed in~\cite{DKM11, DKM12} to even dimensions. 
However, see~\cite{CKLS1, CKLS2} for an application of the  exterior energy estimate in four dimensions restricted to data $(f,0)$ in the context of equivariant wave maps. 

\medskip

In order to salvage some aspect of~\eqref{DKMext} in even dimensions, we show in Section~\ref{conc} that at least a {\em delayed} exterior energy estimate holds.
This is natural in view of two facts: 
\begin{itemize}
\item energy equipartition
\item at least one of the Cauchy data $(f,0)$ or $(0,g)$ is favorable in each even dimension
\end{itemize}
The equipartition property here refers to the fact that after some time, which of course depends on the solution, the energy will split more or less evenly between $\nabla u$ and~$\partial_t u$.  
``Delayed'' refers to lifting the forward (say) light-cone upwards by a certain amount. Equivalently, it means calculating the energy over $|x|\ge t-T$ instead of $|x|\ge t$
for some~$T>0$. 
Figure~\ref{fig:1} shows the distinction between an exterior region both without  and with a time delay.

%%%%%%%%%% Figure 1 %%%%%%%%%%
\begin{figure}
\begin{tikzpicture}[
	>=stealth',
	axis/.style={semithick, ->},
	coord/.style={dashed, semithick},
	yscale = 2,
	xscale = 2]
	\newcommand{\xmin}{-3};
	\newcommand{\xmax}{3};
	\newcommand{\ymin}{-0.2};
	\newcommand{\ymax}{3};
	\newcommand{\delay}{1.5};
	\newcommand{\fsp}{0.2};
	\filldraw[color=light-gray1] ({\xmin+\fsp},0) -- ({\xmax-\fsp},0) -- ({\xmax-\fsp},{\xmax-\fsp}) -- (0,0) -- ({\xmin+\fsp},{abs(\xmin+\fsp)+0});
	\filldraw[color=light-gray2] (-\delay,\delay) -- (\delay,\delay) -- (\xmax-\fsp,{\xmax-\fsp}) -- ({\xmax-\fsp}, \xmax)  -- ({\xmax-\delay},\xmax) -- (0,\delay) -- ({\xmin+\delay},{abs(\xmin)}) -- ({\xmin+\fsp},{abs (\xmin)}) -- ({\xmin+\fsp},{abs(\xmin+\fsp)});
	\draw [axis] (\xmin,0) -- (\xmax,0) node [below left] {$x$};
	\draw [axis] (0,\ymin) -- (0,\ymax+\fsp) node [below right] {$t$};
	\draw [thick] plot[domain={\xmin+\delay}:{\xmax-\delay}] (\x,{(abs(\x)+\delay});
	\draw [thick] plot[domain={\xmin+\fsp}:{\xmax-\fsp}] (\x,{abs(\x)});
	\draw (\xmax-0.9,\xmax-0.9) node [below right] {$|x| \ge t$};
	\draw ({\xmax-1.6},{\xmax-1.6+\delay}) node [below right]  {$|x| \ge t-T^*$};
	\draw (0,\delay) node [below right] {$T^*$};
\end{tikzpicture}
\caption{\label{fig:1} Exterior region without and with time delay}
\end{figure}
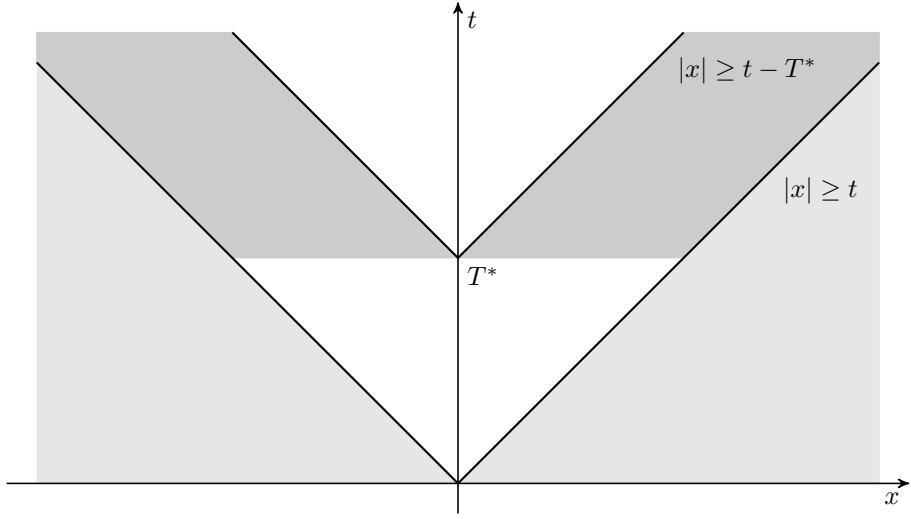
%%%%%%%%%%%%%%%%%%%%%%%

The choice of this $T$ is a delicate matter and depends on the data $(f,g)$.  
The following proposition expresses our main  
quantitative energy evacuation result.   In odd dimensions, results of this nature are obtained via the sharp Huygens principle
and are simpler to obtain. The novelty here lies again with {\em even} dimensions. 

\begin{prop}\label{prop:inner part 0}
For all $\e>0$  and for all $(f,g) \in ( \dot H^1 \times L^2)(\R^d)$ radial  there exists $T=T(\e,f,g,d)>0$ such that 
\begin{equation} 
\label{in_disp}
\| \nabla_{t,x} S(t)(f,g) \|_{L^2(|x| \le t - T)}^2  \le \e \| (f,g) \|_{\dot H^1 \times L^2}^2
\end{equation}
for all $t\ge T$. Equivalently, 
\[
\| \nabla_{t,x} S(t)(f,g) \|_{L^2(|x| > t - T)}^2  \ge (1- \e) \| (f,g) \|_{\dot H^1 \times L^2}^2
\]
for all $t\ge T$.
\end{prop}

In combination with finite propagation speed, Proposition~\ref{prop:inner part 0} implies the following result on the concentration 
of energy near the light-cone. Such statements are well-known in odd dimensions, see \cite[Lemma 4.1]{DKM11} for the three-dimensional version. 

\begin{thm}\label{thm6}
Let $(f,g) \in (\dot H^1 \times L^2)(\R^{d})$ be radial. Then we have the following vanishing of the energy away from the forward light-cone $\{|x|=t\ge0\}$:
\begin{equation}
\nn
\lim_{T \to +\infty} \limsup_{t \to +\infty} \| \nabla_{t,x} S(t)(f,g) \|_{L^2(||x|-t| \ge T )}  =0.
\end{equation}
\end{thm}

Finally, in Section~\ref{bilin} we present various technical results connected with the profile decomposition of Bahouri-G\'erard~\cite{BG98}. 
These are even-dimensional versions of important devices required by~\cite{DKM11, DKM12}. See our followup
work~\cite{CKLS1, CKLS2} with A.~Lawrie for concrete applications of these results. 

 \section{Asymptotic representation of the exterior energy}
\label{basic}

The goal here is to prove the expression \eqref{eq:ext energy} for the asymptotic exterior energy
 in even dimensions, as well as the exterior energy estimate on the region~$\{|x|>|t|\}$.  We shall also contrast
 this to the analogous known results in odd dimensions.  
Denote by $u$ the solution to the \emph{linear} wave equation \eqref{eq:lw} with initial data $(f,g)$ at time 0:
\[ u(t) = S(t)(f,g). \]
Given a set $\calS \subset \m R^d$ (possibly depending on time), define the localized energy functional on~$\calS$ as  
\[ \| (u,v) \|_{\dot H^1 \times L^2(\calS)}^2 := \int_{x \in \calS} \frac12(| v|^2 + |\nabla u|^2)(x)\,  dx. \]
We shall make frequent use of the monotonicity of the energy on outer cones, i.e., the fact that
\[
\| (u,\partial_t u)(t) \|_{\dot H^1 \times L^2(|x| \ge t-T)} =: \| \nabla_{t,x} u(t) \|_{L^2(|x| \ge t-T)} \le \| \nabla_{t,x} u(s) \|_{L^2(|x|\ge s-T)}
\]
for all $T \le s \le t$. In this section, $T=0$. 

\begin{proof}[Proof of Theorem~\ref{th1}]
The solution is given by 
\[
u(t) = \cos(t|\nabla|)f+\frac{\sin(t|\nabla|)}{|\nabla|} g.
\]
Let  $\hat{f},\hat{g}$ be the Fourier transforms in~$\R^d$:
\[
\hat{f}(\xi)=\int_{\R^d} e^{-ix\cdot \xi} f(x) \, dx,\qquad f(x)=(2\pi)^{-d} \int_{\R^d} e^{ix\cdot \xi}  \hat{f}(\xi)\, d\xi.
\] 
For radial functions, $\hat{f}$ is again radial. Recall that 
\[
\widehat{\sigma_{S^{d-1}} } (\xi) = (2\pi)^{\frac{d}{2}} |\xi|^{-\nu} J_\nu(|\xi|),\quad \nu: = \frac{d-2}{2}\ge0
\]
where 
$J_\nu$ is the Bessel function of the first type of order~$\nu$. 
It is characterized as being the solution of 
\begin{equation} \label{def:bessel}
x^2 J_{\nu}''(x) + x J_\nu'(x) + (x^2 - \nu^2) J_\nu(x) =0
\end{equation}
which is regular at $x=0$ (unique up to a multiplicative constant). 
The inversion formula takes the form
\[
f(r) =  (2\pi)^{-\frac{d}{2}} \int_0^\infty \hat{f}(\rho) J_\nu(r\rho) (r\rho)^{-\nu} \ro^{d-1}\, d\ro 
\]
The Plancherel identity takes the form $\|\hat{f}\|_2^2=(2\pi)^d \|f\|_2^2$.
For the solution $u(t,r)$ this means that 
\begin{align*}
u(t,r) & =  (2\pi)^{-\frac{d}{2}} \int_0^\infty \left( \cos(t\rho) \hat f(\rho) + \frac{\sin(t \rho)}{\rho} \hat g(\rho) \right) J_{\nu}(r \rho) (r\rho)^{-\nu}  \rho^{d-1} \, d\rho \\
\partial_t u(t,r) & =  (2\pi)^{-\frac{d}{2}} \int_0^\infty \left(-\sin (t\rho) \rho \hat f(\rho) + \cos(t \rho) \hat g(\rho) \right) J_{\nu}(r \rho) (r\rho)^{-\nu}  \rho^{d-1} \,  d\rho. 
\end{align*}
We shall invoke the standard asymptotics for the Bessel functions, see~\cite{AS}, 
\begin{equation}\label{eq:Jnu exp}
\begin{aligned}
J_\nu(x) = \sqrt{\frac{2}{\pi x} } \left[ (1+ \w_2(x)) \cos (x-\tau) + \w_1(x) \sin(x-\tau) \right], \\
J_\nu'(x) = \sqrt{\frac{2}{\pi x} } \left[ \tilde \w_1(x) \cos(x-\tau) - (1+\tilde \w_2(x)) \sin(x-\tau) \right].
\end{aligned}
\end{equation}
with phase-shift $\tau = (d-1)\frac{\pi}{4}$, and with the bounds (for $n \ge 0$, $x \ge 1$)
\begin{equation}
| \w_1^{(n)}(x)| + |\tilde \w_1^{(n)}(x)| \le C_n\, x^{-1-n}, \quad | \w_2^{(n)}(x)| + |\tilde \w_2^{(n)}(x)| \le C_n \, x^{-2-n}. \label{est:w}
\end{equation}
 Moreover, it suffices to let $f,g$ be Schwartz functions by energy bounds, and we may assume that $\hat{f}(\rho)$ and $\hat{g}(\rho)$
 are supported on $0<\rho_{*}<\rho<\rho^{*}<\infty$. 
We begin with the kinetic part of the outer energy, viz. 
\begin{equation}
\label{Ekinetic}
\begin{aligned} 
\MoveEqLeft 
(2\pi)^d    \frac12  \| \partial_t u(t) \|_{L^2(|x| \ge t)}^2  = (2\pi)^d|\m S^{d-1}|\int_{t}^\infty \frac12 |\partial_t u(t,r)|^2 \, r^{d-1} \, dr \\
&= (2\pi)^d |\m S^{d-1}| \lim_{\eps\to0+} \int_{t}^\infty \frac12 |\partial_t u(t,r)|^2  r^{d-1} e^{-\eps r} \, dr \\
& = \lim_{\eps\to0+}\int_t^\infty \iint  \frac12 \left(-\sin (t\rho_1) \rho_1 \hat f(\rho_1) + \cos(t \rho_1) \hat g(\rho_1) \right) \\
& \qquad \qquad \cdot \left(-\sin (t\rho_2) \rho_2 \ol{ \hat f(\rho_2)} + \cos(t \rho_2) \ol{\hat g(\rho_2)} \right) \\
& \qquad \qquad \cdot  J_{\nu}(r \rho_1)  J_{\nu}(r \rho_2) (r^2 \rho_1 \rho_2)^{-\nu}  (\rho_1 \rho_2)^{d-1} \, d\rho_1 d\rho_2\,  r^{d-1}\, e^{-\eps r} \, dr.
\end{aligned}
\end{equation}
For each $\eps>0$ fixed, the integrals here are absolutely convergent.  
In view of the asymptotic expansion of the Bessel functions as stated above, 
the leading term for~\eqref{Ekinetic}   is given by the following expression, with $\mu=\nu+\f12=\f{d-1}{2}$:  
\begin{equation}
\label{Idef}
\begin{aligned}
 &{\frac{1}{\pi}} \lim_{\eps\to0+}\int\limits_t^\infty \int\limits_0^\infty \int\limits_0^\infty 
\left(-\sin (t\rho_1) \rho_1 \hat f(\rho_1) + \cos(t \rho_1) \hat g(\rho_1) \right) \\
& \qquad \qquad \cdot \left(-\sin (t\rho_2) \rho_2 \ol{\hat f(\rho_2)} + \cos(t \rho_2) \ol{ \hat g(\rho_2)} \right) \\
& \qquad \qquad \cdot \cos (r \rho_1 -\tau)  \cos (r \rho_2-\tau) (\rho_1 \rho_2)^\mu  \, d\rho_1 d\rho_2 \, e^{-\eps r}\, dr.
\end{aligned}
\end{equation}
We shall show later that this indeed captures the correct asymptotic behavior of the exterior kinetic energy. To be specific, 
we make the following claim: 
\begin{equation}
\label{Eeven}
\begin{aligned} 
&
(2\pi)^d |\m S^{d-1}|^{-1}  \big(  \| \partial_t u(t) \|_{L^2(|x| \ge t)}^2 + \| \partial_r u(t) \|_{L^2(|x| \ge t)}^2\big) \\
& = \frac{2}{\pi}\lim_{\eps\to0+}\int_t^\infty \iint   \left(-\sin (t\rho_1) \rho_1 \hat{f}(\rho_1) + \cos(t \rho_1) \hat g(\rho_1) \right) \\
& \qquad \qquad \cdot \left(-\sin (t\rho_2) \rho_2 \ol{ \hat{f}(\rho_2)} + \cos(t \rho_2) \ol{\hat g(\rho_2)} \right) \\
& \qquad \qquad \cdot  \cos(r \rho_1-\tau)  \cos(r \rho_2-\tau) (\rho_1 \rho_2)^{\mu}   \, d\rho_1 d\rho_2\,  \, e^{-\eps r} \, dr \\
&\quad + \frac{2}{\pi}\lim_{\eps\to0+}\int_t^\infty \iint   \left(\cos (t\rho_1) \rho_1 \hat{f}(\rho_1) + \sin(t \rho_1) \hat g(\rho_1) \right) \\
& \qquad \qquad \cdot \left(\cos(t\rho_2) \rho_2 \ol{ \hat{f}(\rho_2)} + \sin(t \rho_2) \ol{\hat g(\rho_2)} \right) \\
& \qquad \qquad \cdot  \sin(r \rho_1-\tau)  \sin(r \rho_2-\tau)   (\rho_1 \rho_2)^{\mu} \, d\rho_1 d\rho_2\,   \, e^{-\eps r} \, dr   + o(1) 
\end{aligned}
\end{equation}
where $o(1)$ is with respect to $t\to\pm\I$.  We added in the contribution by~$\p_r u$:
\EQ{\label{u_r}
\partial_r u(t,r)  &=  (2\pi)^{-\frac{d}{2}} \int_0^\infty \left(-\sin (t\rho) \rho \hat f(\rho) + \cos(t \rho) \hat g(\rho) \right) \big(\rho J_{\nu}'(r \rho) -\\
&\qquad\qquad \nu r^{-1}J_{\nu}(r \rho)\big) (r\rho)^{-\nu}  \rho^{d-1} \,  d\rho
}
where $r^{-1}J_{\nu}(r \rho)$ will be seen to be an error term. 
We now proceed to extract~\eqref{eq:ext energy} from the integrals in~\eqref{Eeven}. 
In order to carry out the $r$-integration in~\eqref{Eeven}, we use (note $2\tau\in (\Z+\frac12)\pi$)
\EQ{ \nn 
\cos(r \rho_1-\tau)  \cos(r \rho_2-\tau) &=\frac12[\cos(r(\rho_1+\rho_2)-2\tau) + \cos(r(\rho_1-\rho_2))]\\
&=\frac12[(-1)^\nu \sin(r(\rho_1+\rho_2)) + \cos(r(\rho_1-\rho_2))] \\
\sin(r \rho_1-\tau)  \sin(r \rho_2-\tau) &=\frac12[-\cos(r(\rho_1+\rho_2)-2\tau) + \cos(r(\rho_1-\rho_2))] \\
&= \frac12[-(-1)^\nu \sin(r(\rho_1+\rho_2)) + \cos(r(\rho_1-\rho_2))]
}
where $\nu=\frac{d-2}{2}$.  In what follows, we slightly abuse notation by writing $\hat{f'}(\ro):=\rho\hat{f}(\ro)$.

For any smooth compactly supported  functions $\phi, \psi $ on $(0,\infty)$, one has for every $t\in\R$
\begin{align}
& \lim_{\eps\to0+} \int_t^\infty \iint \cos(r(\rho_1 - \rho_2)) \phi(\rho_1) \psi(\rho_2) \, e^{-\eps r}dr \, d\rho_1 d\rho_2 \nonumber \\
&\qquad = \pi
\int \phi(\rho) \psi(\rho) d\rho - \iint \frac{\sin(t(\rho_1 - \rho_2))}{\rho_1-\rho_2} \phi(\rho_1) \psi(\rho_2) \, d\rho_1 d\rho_2 \label{ker1} \\
& \lim_{\eps\to0+}  \int_t^\infty \iint  \sin (r(\rho_1 + \rho_2))  \phi(\rho_1) \psi(\rho_2)  e^{-\eps r}dr\,  d\rho_1 d\rho_2 \nonumber \\
& \qquad = \iint \frac{\cos (t(\rho_1 + \rho_2))}{\rho_1+\rho_2} \phi(\rho_1) \psi(\rho_2) \, d\rho_1 d\rho_2. \label{ker2}
\end{align}
To prove~\eqref{ker1} we note that
\begin{align*}
 \lim_{\eps\to0+} \int_t^\infty  \cos(a r) e^{-\eps 	r}\, dr & = \lim_{\eps\to0+} \frac12 \Big( -\frac{e^{t(ia-\eps)} }{ia-\eps}  +  \frac{e^{-t(ia+\eps)} }{ia+\eps} \Big)  
= \pi \delta_{0}(a) - \frac{\sin(ta)}{a}
\end{align*}
where the limit is to be taken in the distributional sense. 
For~\eqref{ker2} the argument is essentially the same. 

Carrying out the $r$-integration using \eqref{ker1}, \eqref{ker2} and ignoring constant prefactors yields:
\EQ{
\nn
&\iint \big[ \cos(t(\ro_1-\ro_2)) (\hat{f'}(\rho_1)  \ol{ \hat{f'}(\rho_2)} + \hat g(\rho_1) \ol{\hat g(\rho_2)} )   -\sin(t(\ro_1-\ro_2)) \cdot \\
& \cdot (\hat{f'}(\rho_1)  \ol{\hat g(\rho_2)}  - \hat g(\rho_1) \ol{ \hat{f'}(\rho_2)} )
\big] \Big( \pi\delta_0(\ro_1-\ro_2)  - \frac{\sin(t(\ro_1-\ro_2))}{\ro_1-\ro_2} \Big) (\ro_1\ro_2)^\mu\, d\ro_1 d\ro_2 \\
&+ (-1)^{\frac{d}{2}}  \iint \big[\cos(t(\ro_1+\ro_2)) (\hat{f'}(\rho_1)  \ol{ \hat{f'}(\rho_2)} - \hat g(\rho_1) \ol{\hat g(\rho_2)} ) + \sin(t(\ro_1+\ro_2))\cdot\\
&\cdot (\hat{f'}(\rho_1)  \ol{\hat g(\rho_2)}  + \hat g(\rho_1) \ol{ \hat{f'}(\rho_2)} )
\big]   \frac{\cos(t(\ro_1+\ro_2))}{\ro_1+\ro_2}   (\ro_1\ro_2)^\mu\, d\ro_1 d\ro_2
}
which further simplifies to (integration extending over $(0,\infty)$) 
\EQ{\label{uff}
&\pi\int_0^\infty (|\hat{f'}(\rho)|^2+|\hat{g}(\rho)|^2)\ro^{d-1}\, d\ro \\
&-\frac12 \iint \frac{\sin(2t(\ro_1-\ro_2))  }{\ro_1-\ro_2} (\hat{f'}(\rho_1)  \ol{ \hat{f'}(\rho_2)} + \hat g(\rho_1) \ol{\hat g(\rho_2)} ) (\ro_1\ro_2)^\mu\, d\ro_1 d\ro_2 \\
&+ \iint \frac{\sin^2(t(\ro_1-\ro_2))    }{\ro_1-\ro_2}    (\hat{f'}(\rho_1)  \ol{\hat g(\rho_2)}  - \hat g(\rho_1) \ol{ \hat{f'}(\rho_2)} )   (\ro_1\ro_2)^\mu\, d\ro_1 d\ro_2 \\
&+(-1)^{\frac{d}{2}} \iint \frac{\cos^2(t(\ro_1+\ro_2))   }{\ro_1+\ro_2} (\hat{f'}(\rho_1)  \ol{ \hat{f'}(\rho_2)} - \hat g(\rho_1) \ol{\hat g(\rho_2)} ) (\ro_1\ro_2)^\mu\, d\ro_1 d\ro_2 \\
&+ \frac{(-1)^\frac{d}{2}}{2}  \iint \frac{\sin(2t(\ro_1+\ro_2))   }{\ro_1+\ro_2}    (\hat{f'}(\rho_1)  \ol{\hat g(\rho_2)}  + \hat g(\rho_1) \ol{ \hat{f'}(\rho_2)} )   (\ro_1\ro_2)^\mu\, d\ro_1 d\ro_2 
}
It remains to determine the limit $t\to\I$. 
First, recall that for any $a>0$ (with $\calF$ denoting the Fourier transform on~$\R$)
\EQ{\label{trig FT}
\calF[\frac{\sin (ax)}{x}](\xi)=\pi \chi_{(-a,a)}(\xi),\qquad \calF[\frac{\cos (ax)}{x}](\xi) = \pi i [-\chi_{(-\infty,-a)}+\chi_{(a,\infty)}] 
}
The integral on the second line is of the form 
\EQ{\label{eq:mid int}
& \int_{\R^{2}}  \frac{\sin(2t(\rho_1-\rho_2))}{\rho_1-\rho_2} \phi(\rho_{1}) \overline{\phi(\rho_{2})}\, d\rho_{1}d\rho_{2} 
 = \frac12 \int_{-2t}^{2t} |\hat{\phi}(\xi)|^{2}\, d\xi 
}
with $\phi(\rho)=\hat{f}(\rho)\rho^{\mu}$.  As $t\to\infty$, this approaches 
\[
\frac12\|\hat{\phi}\|_2^2=\pi \|\phi\|_2^2.
\]
Hence, the sum of the first and second lines in~\eqref{uff} tends to 
\[
\f{\pi}{2}\int_0^\infty (|\hat{f'}(\rho)|^2+|\hat{g}(\rho)|^2)\ro^{d-1}\, d\ro \qquad\text{\ \ as\ \ }t\to\I
\] 
Integration by parts shows that the fifth line  vanishes in the limit $t\to\infty$   (the data are Schwartz). 
For the expressions on the third and fourth lines, respectively, we use
\[
\cos^2(x)=\f12 + \f12 \cos(2x),\qquad \sin^2(x)=\f12-\f12 \cos(2x)
\]
and \eqref{trig FT}, \eqref{eq:mid int} to deduce that in the limit of~\eqref{uff} as $t\to\I$ equals
\EQ{\nn
&\frac{\pi}{2} \int  (|\hat{f'}(\rho)|^2+|\hat{g}(\rho)|^2)\ro^{d-1}\, d\ro \\
&+\frac{(-1)^{\frac{d}{2}} }{2} \iint \frac{1 }{\ro_1+\ro_2} (\hat{f'}(\rho_1)  \ol{ \hat{f'}(\rho_2)} - \hat g(\rho_1) \ol{\hat g(\rho_2)} ) (\ro_1\ro_2)^\mu\, d\ro_1 d\ro_2 \\
&+ \Re \iint \frac{1  }{\ro_1-\ro_2}    \hat{f'}(\rho_1)  \ol{\hat g(\rho_2)}      (\ro_1\ro_2)^\mu\, d\ro_1 d\ro_2
}
up to a constant prefactor, and with integration extending over $(0,\infty)$. This is exactly what~\eqref{eq:ext energy} claims. 

It remains to verify the claimed dominance of the leading order terms of the Bessel expansion, see~\eqref{Eeven}. 
For simplicity, we restrict ourselves to the kinetic energy~\eqref{Ekinetic}. 
Subtracting~\eqref{Idef} from~\eqref{Ekinetic}  yields, with $\omega_j$ as in~\eqref{est:w}, 
\begin{align*}
&  {\frac{2}{\pi}} \lim_{\eps\to0+}\int\limits_t^\infty \int\limits_0^\infty \int\limits_0^\infty 
\left(-\sin (t\rho_1) \rho_1 \hat f(\rho_1) + \cos(t \rho_1) \hat g(\rho_1) \right) \\
& \qquad \qquad \cdot \left(-\sin (t\rho_2) \rho_2 \ol{\hat f(\rho_2)} + \cos(t\rho_2) \ol{ \hat g(\rho_2)} \right) \\
& \qquad \qquad\cdot \big[ \big(\omega_2(r\rho_1) + \omega_2(r\rho_2) + \omega_2(r\rho_1) \omega_2(r\rho_2) \big)   \cos (r \rho_1 -\tau)  \cos (r \rho_2-\tau) \\
&\qquad\qquad\quad  + \omega_1(r\rho_1) (1+\omega_2(r\rho_2))
\sin (r \rho_1 -\tau)  \cos (r \rho_2-\tau) \\
& \qquad\qquad\quad + \omega_1(r\rho_2) (1+\omega_2(r\rho_1))
\sin (r \rho_2 -\tau)  \cos (r \rho_1-\tau) \\
&\qquad\qquad \quad + \omega_1(r\rho_1) \omega_1(r\rho_2) \sin(r\rho_1-\tau)\sin(r\rho_2-\tau)\big]  (\rho_1 \rho_2)^\mu  \, d\rho_1 d\rho_2 \, e^{-\eps r}\, dr.
\end{align*}
All terms here are treated in a similar fashion. As a representative example, consider for all $\eps>0$
the error term
\begin{multline*}
E_1(\eps):= \int\limits_t^\infty \int\limits_0^\infty \int\limits_0^\infty \sin (t \rho_1) \sin (t \rho_2) \cos(r\rho_1 - \tau) \sin (r\rho_2-\tau) \w_1(r\rho_2) \\
\cdot \hat f(\rho_1) \overline{\hat f(\rho_2)} (\rho_1 \rho_2)^{\mu+1} e^{-\eps r} \, d\rho_1 d\rho_2 dr,
\end{multline*}
As before, we write
\[ \cos(r\rho_1 - \tau) \sin (r\rho_2-\tau) = -\frac{1}{2} \big[ (-1)^\nu \cos(r (\rho_1+\rho_2)) + \sin(r(\rho_1- \rho_2)) \big], \]
expand the trigonometric functions on the right-hand side into complex exponentials, 
and perform an integration by parts in the $r$ variable as follows: for any $\sigma\in\R$ and dropping the subscripts on $\omega,\rho$ for simplicity,  one has 
\begin{equation}
\label{IBP}
\int_t^\infty e^{-[\eps\mp i\sigma]r} \, \omega(r\rho)\, dr = \frac{e^{-[\eps\mp i\sigma ]t}}{\eps \mp i\sigma} \omega(t\rho)
 + \int_t^\infty \frac{e^{-[\eps\mp i\sigma ]r}}{\eps \mp i\sigma} \omega'(r\rho)\rho \, dr.
\end{equation}
We apply this with $\sigma=\rho_1+\rho_2$ and $\sigma=\rho_1-\rho_2$ to the fully expanded form of $E_1(\eps)$ as explained above. 
In both cases one has the uniform bounds
\[
\sup_{\eps>0}\Big\| \int_{-\infty}^\infty \frac{\phi(\rho_2)}{(\rho_1 \pm \rho_2) \pm i\eps} \, d\rho_2\Big\|_{L^2(\rho_1)} \le C\|\phi\|_2.
\]
In order to use this, we distribute the exponential factors as well as all  weights over the functions $\hat f(\rho_1)$ and $ \overline{\hat f(\rho_2)} $, respectively.
For the first term on the right-hand side of~\eqref{IBP} we then obtain an estimate $O(t^{-1})$ from the decay of the weight~$\omega$, whereas for the integral in~\eqref{IBP}
we obtain a $O(r^{-2})$-bound via 
\[
\sup_{\rho>0} |\omega'(r\rho)\rho^2| \le C\, r^{-2}
\]
which then leads to the final bound
\[
\int_t^\infty O(r^{-2})\, dr = O(t^{-1}).
\]
The $O$-here are   uniform in~$\eps>0$.  Note that various $\rho$-factors which are introduced by the $\omega$-weights  are
harmless due to our standing assumption that $0<\rho_*<\rho <\rho^*$. 

All error terms fall under this scheme. In fact, those involving two $\omega$-factors yield a $O(t^{-2})$-estimate. This concludes the proof.
\end{proof}

As an immediate corollary one obtains the exterior energy estimate in even dimensions. 

\begin{cor}
\label{cor:lin2}
 Let $d\ge2$ be even. 
Let $\Box u =0$, $u(0)=f\in \dot H^1(\R^d)$ be radial, $\dot u(0)=0$. Then for all $t\ge0$,  and provided $d\equiv 0 \mod 4$, 
\EQ{\label{f0bd}
\| \nabla_{t,x} S(t)(f,0) \|_{L^2(r\ge t)}^{2}  \ge c(d)\, \|\nabla f \|_{L^2}^{2}
}
where $c(d)>0$ is an absolute constant that only depends on the dimension.  If $d\equiv 2 \mod 4$ then {\em there can be no estimate}
of the form~\eqref{f0bd}
for all $t\ge0$.  For the dual initial value problem $\Box u =0$, $u(0)=0$, $\dot u(0)=g\in L^2(\R^d)$  radial, one has
\[
\| \nabla_{t,x} S(t)(0,g) \|_{L^2(r\ge t)}^{2} \ge c(d)\, \|g\|_2^{2}
\]
if  $d\equiv 0 \mod 4$, whereas it fails if $d\equiv 2 \mod 4$.
\end{cor}

\begin{proof} 
Denote the left-hand side of~\eqref{f0bd} by $E(t)$. Since $E(t)$ is decreasing, it suffices to consider the limit~$t\to\I$. 
Let us fix dimensions $d=4,8,12,\ldots$ and data $(f,0)$. Then~\eqref{eq:ext energy}  implies that 
\EQ{
\lim_{t\to\I} C(d) \| \nabla_{t,x} S(t)(f,0) \|_{L^2(r\ge t)}^{2}  =  \f{\pi}{2} \| \nabla f\|_2^2  + \f12 \big\lan H( \rho^{\f{d+1}{2}} \hat{f}), \rho^{\f{d+1}{2}} \hat{f}\big\ran 
}
It is well-known that the Hankel transform $H$ is a positive operator on~$L^2((0,\I))$, since $H=\cL^2$ where $\cL$ is the 
Laplace transform which is self-adjoint. See for example~\cite[Section~16.3.3]{LaxFA02}. 

The failure of the estimate for $d=2,6,10, \ldots$ and data $(f,0)$ follows just as easily since the operator norm of~$\cL$ on~$L^2$ equals~$\sqrt{\pi}$. 
The Cauchy problem with data $(0,g)$ is treated analogously. 
\end{proof}

For the sake of completeness, we contrast the even-dimensional case of Theorem~\ref{th1} with the odd-dimensional one of Theorem~\ref{th2}. 
The asymptotic calculations are completely analogous to the ones above, with the dimension entering only (in an essential way) through the phase-shift~$\tau=\frac{d-1}{4}\pi$ 
in the expansions of the Bessel functions for large arguments.  The key feature being that~$2\tau$ is an integer if $d$ is odd, and a half-integer otherwise. 

\begin{proof}[Proof of Theorem~\ref{th2}]
We begin by computing the asymptotic form of the exterior energy as in even dimensions, say for $t\ge0$. With all Fourier transforms being those in~$\R^d$, one has 
\begin{equation}
\label{Eodd}
\begin{aligned} 
&
(2\pi)^d   \big(  \| \partial_t u(t) \|_{L^2(|x| \ge t)}^2 + \| \partial_r u(t) \|_{L^2(|x| \ge t)}^2\big) \\
& = \frac{2}{\pi}\lim_{\eps\to0+}\int_{t}^\infty \iint   \left(-\sin (t\rho_1) \rho_1 \hat{f}(\rho_1) + \cos(t \rho_1) \hat g(\rho_1) \right) \\
& \qquad \qquad \cdot \left(-\sin (t\rho_2) \rho_2 \ol{ \hat{f}(\rho_2)} + \cos(t \rho_2) \ol{\hat g(\rho_2)} \right) \\
& \qquad \qquad \cdot  \cos(r \rho_1-\tau)  \cos(r \rho_2-\tau) (r^2 \rho_1 \rho_2)^{\mu}   \, d\rho_1 d\rho_2\,  \, e^{-\eps r} \, dr+\\
&+ \frac{2}{\pi}\lim_{\eps\to0+}\int_{t}^\infty \iint   \left(\cos (t\rho_1) \rho_1 \hat{f}(\rho_1) + \sin(t \rho_1) \hat g(\rho_1) \right) \\
& \qquad \qquad \cdot \left(\cos(t\rho_2) \rho_2 \ol{ \hat{f}(\rho_2)} + \sin(t \rho_2) \ol{\hat g(\rho_2)} \right) \\
& \qquad \qquad \cdot  \sin(r \rho_1-\tau)  \sin(r \rho_2-\tau)   (\rho_1 \rho_2)^{\mu} \, d\rho_1 d\rho_2\,   \, e^{-\eps r} \, dr + o(1)
\end{aligned}
\end{equation}
where the $o(1)$ is for $t\to\infty$.  Here $\tau=\frac{d-1}{4}\pi$, $\mu=\frac{d-1}{2}$. Moreover, we used the asymptotic expansions of the Bessel functions~\eqref{eq:Jnu exp}, and we absorbed all error terms
in the~$o(1)$, which is justified by the exact same reasoning as in the proof of~\eqref{eq:ext energy}. 
In order to carry out the $r$-integration, we use (note $2\tau\in \Z\pi$)
\EQ{ \nn 
\cos(r \rho_1-\tau)  \cos(r \rho_2-\tau) &=\frac12[\cos(r(\rho_1+\rho_2)-2\tau) + \cos(r(\rho_1-\rho_2))]\\
&=\frac12[(-1)^\mu \cos(r(\rho_1+\rho_2)) + \cos(r(\rho_1-\rho_2))] \\
\sin(r \rho_1-\tau)  \sin(r \rho_2-\tau) &=\frac12[-\cos(r(\rho_1+\rho_2)-2\tau) + \cos(r(\rho_1-\rho_2))] \\
&= \frac12[-(-1)^\mu \cos(r(\rho_1+\rho_2)) + \cos(r(\rho_1-\rho_2))]
}
In what follows, we slightly abuse notation by writing $\hat{f'}(\ro):=\rho\hat{f}(\ro)$.
Carrying out the $r$-integration using \eqref{ker1}, \eqref{ker2} and applying trigonometric identities yields (ignoring constant prefactors):
\begin{align*}
&\iint \big[ \cos(t(\ro_1-\ro_2)) (\hat{f'}(\rho_1)  \ol{ \hat{f'}(\rho_2)} + \hat g(\rho_1) \ol{\hat g(\rho_2)} ) -\sin(t(\ro_1-\ro_2)) \\
& (\hat{f'}(\rho_1)  \ol{\hat g(\rho_2)}  - \hat g(\rho_1) \ol{ \hat{f'}(\rho_2)} )
\big] \Big( \pi\delta_0(\ro_1-\ro_2)  - \frac{\sin(t(\ro_1-\ro_2))}{\ro_1-\ro_2} \Big) (\ro_1\ro_2)^\mu\, d\ro_1 d\ro_2 \\
&+ (-1)^\mu  \iint \big[\cos(t(\ro_1+\ro_2)) (\hat{f'}(\rho_1)  \ol{ \hat{f'}(\rho_2)} - \hat g(\rho_1) \ol{\hat g(\rho_2)} ) \\
 &+ \sin(t(\ro_1+\ro_2)) (\hat{f'}(\rho_1)  \ol{\hat g(\rho_2)}  + \hat g(\rho_1) \ol{ \hat{f'}(\rho_2)} )
\big]   \frac{\sin(t(\ro_1+\ro_2))}{\ro_1+\ro_2}   (\ro_1\ro_2)^\mu\, d\ro_1 d\ro_2
\end{align*}
which further simplifies to (integration extending over $(0,\infty)$) 
\EQ{\nn
&\pi\int_0^\infty (|\hat{f'}(\rho)|^2+|\hat{g}(\rho)|^2)\ro^{d-1}\, d\ro \\
&-\frac12 \iint \frac{\sin(2t(\ro_1-\ro_2)) }{\ro_1-\ro_2} (\hat{f'}(\rho_1)  \ol{ \hat{f'}(\rho_2)} + \hat g(\rho_1) \ol{\hat g(\rho_2)} ) (\ro_1\ro_2)^\mu\, d\ro_1 d\ro_2 \\
&+\frac12 \iint \frac{1 - \cos(2t(\ro_1-\ro_2))  }{\ro_1-\ro_2}    (\hat{f'}(\rho_1)  \ol{\hat g(\rho_2)}  -   \hat g(\rho_1) \ol{ \hat{f'}(\rho_2)} )   (\ro_1\ro_2)^\mu\, d\ro_1 d\ro_2 \\
&+\frac{(-1)^\mu}{2} \iint \frac{\sin(2t(\ro_1+\ro_2))   }{\ro_1+\ro_2} (\hat{f'}(\rho_1)  \ol{ \hat{f'}(\rho_2)} -   \hat g(\rho_1) \ol{\hat g(\rho_2)} ) (\ro_1\ro_2)^\mu\, d\ro_1 d\ro_2 \\
&+ \frac{(-1)^\mu}{2}  \iint \frac{1 - \cos(2t(\ro_1+\ro_2))  }{\ro_1+\ro_2}    (\hat{f'}(\rho_1)  \ol{\hat g(\rho_2)}  +  \hat g(\rho_1) \ol{ \hat{f'}(\rho_2)} )   (\ro_1\ro_2)^\mu\, d\ro_1 d\ro_2 
}
We may now pass to the limit $t\to\infty$.
The terms involving $\sin(2t(\ro_1+\ro_2))$ and $\cos(2t(\ro_1+\ro_2))$ in
the fourth and fifth lines, respectively, vanish in the limit $t\to\infty$ as can be seen by integration by parts (we may again assume that the data are Schwartz).
The asymptotic form of the terms
 involving $\sin(2t(\ro_1-\ro_2))$ and $\cos(2t(\ro_1-\ro_2))$ in
the second and third lines, respectively, follows from~\eqref{trig FT}:
\EQ{\nn 
& \lim_{t\to\infty} \iint \frac{\sin(2t(\ro_1-\ro_2)) }{\ro_1-\ro_2} (\hat{f'}(\rho_1)  \ol{ \hat{f'}(\rho_2)} + \hat g(\rho_1) \ol{\hat g(\rho_2)} ) (\ro_1\ro_2)^\mu\, d\ro_1 d\ro_2\\
&= \pi\int  (|\hat{f'}(\rho)|^2+|\hat{g}(\rho)|^2)\ro^{d-1}\, d\ro
}
and
\[
\lim_{t\to\infty} \iint \frac{  \cos(2t(\ro_1-\ro_2))  }{\ro_1-\ro_2}    (\hat{f'}(\rho_1)  \ol{\hat g(\rho_2)}  - \hat g(\rho_1) \ol{ \hat{f'}(\rho_2)} )   (\ro_1\ro_2)^\mu\, d\ro_1 d\ro_2 =0
\]
In conclusion, we obtain the following asymptotic expression for the left-hand side of~\eqref{Eeven} for~$d$ odd as~$t\to\pm\infty$: 
\EQ{\label{odd extE}
&\frac{\pi}{2}\int  (|\hat{f'}(\rho)|^2+|\hat{g}(\rho)|^2)\ro^{d-1}\, d\ro \\
&\pm \Re \iint \Big[ \frac{1  }{\ro_1-\ro_2} + (-1)^\mu\frac{1  }{\ro_1+\ro_2}   \Big]  \hat{f'}(\rho_1)  \ol{\hat g(\rho_2)}     (\ro_1\ro_2)^\mu\, d\ro_1 d\ro_2 
}
up to a constant prefactor, and with integration extending over $(0,\infty)$. 
This is exactly~\eqref{eq:ext energy odd}. 
\end{proof}

In order to deduce~\eqref{DKMext} from~\eqref{odd extE}, one chooses the direction of time so as the make the second line of~\eqref{odd extE} nonnegative.  

\section{Delayed exterior energy and energy concentration}
\label{conc}

We now turn to a delayed version of the exterior energy bound. 
We will rely on the radial Fourier formalism from the proof of Theorem~\ref{th1} without further mention. 

\begin{proof}[Proof of Proposition~\ref{prop:inner part 0}]
Denote by $u(t,x) = S(t)(f,g)$ the solution of the wave equation~\eqref{eq:lw} as above.
We first remark that by conservation of energy~\eqref{in_disp} is equivalent to the following: 
\begin{equation} \label{out_energy0}
\| (f,g) \|_{\dot H^1 \times L^2}^2  - \| \nabla_{t,x} u \|_{L^2( |x| \ge t -T)}^2  \le \e \| (f,g) \|_{\dot H^1 \times L^2}^2 
\end{equation}
for all $t\ge T$ where $T=T(\eps,f,g,d)$. Due to the fact that $$t\mapsto \| \nabla_{t,x} u(t)\|_{L^2(|x| \ge t -T)}$$ 
is monotone decreasing, we see that~\eqref{out_energy0} is a consequence of the following bound
\begin{equation} %\label{out_energy}
\limsup_{t \to +\infty} \Big[ \| (f,g) \|_{\dot H^1 \times L^2}^2  - \| \nabla_{t,x} u \|_{L^2( |x| \ge t -T)}^2 \Big]  \le \eps  \| (f,g) \|_{\dot H^1 \times L^2}^2 
\end{equation}
which we now prove. 
 Moreover, it suffices to let $f,g$ be Schwartz functions by energy bounds, and we may assume that $\hat{f}(\rho)$ and $\hat{g}(\rho)$
 are supported on $0<\rho_{*}<\rho<\rho^{*}<\infty$. 
We begin with the kinetic part of the outer energy, viz. 
\begin{equation}
\label{Ekinetic2}
\begin{aligned} 
\MoveEqLeft 
(2\pi)^d    \frac12  \| \partial_t u(t+T) \|_{L^2(|x| \ge t)}^2  = (2\pi)^d|\m S^{d-1}|\int_{t}^\infty \frac12 |\partial_t u(t+T,r)|^2 \, r^{d-1} \, dr \\
&= (2\pi)^d |\m S^{d-1}| \lim_{\eps\to0+} \int_{t}^\infty \frac12 |\partial_t u(t+T,r)|^2  r^{d-1} e^{-\eps r} \, dr \\
& = \lim_{\eps\to0+}\int_t^\infty \iint  \frac12 \left(-\sin ((t+T)\rho_1) \rho_1 \hat f(\rho_1) + \cos((t+T) \rho_1) \hat g(\rho_1) \right) \\
& \qquad \qquad \cdot \left(-\sin ((t+T)\rho_2) \rho_2 \ol{ \hat f(\rho_2)} + \cos((t+T) \rho_2) \ol{\hat g(\rho_2)} \right) \\
& \qquad \qquad \cdot  J_{\nu}(r \rho_1)  J_{\nu}(r \rho_2) (r^2 \rho_1 \rho_2)^{-\nu}  (\rho_1 \rho_2)^{d-1} \, d\rho_1 d\rho_2\,  r^{d-1}\, e^{-\eps r} \, dr.
\end{aligned}
\end{equation}
For each $\eps>0$ fixed, the integrals here are absolutely convergent. 
In view of the asymptotic expansion of the Bessel functions~\eqref{eq:Jnu exp}, 
the leading term for~\eqref{Ekinetic}   is given by the following expression, where $\mu=\nu+\f12=\f{d-1}{2}$: 
\begin{equation}
\label{Idef2}
\begin{aligned}
I(T,t) & : = {\frac{1}{\pi}} \lim_{\eps\to0+}\int\limits_t^\infty \int\limits_0^\infty \int\limits_0^\infty 
\left(-\sin ((t+T)\rho_1) \rho_1 \hat f(\rho_1) + \cos((t+T) \rho_1) \hat g(\rho_1) \right) \\
& \qquad \qquad \cdot \left(-\sin ((t+T)\rho_2) \rho_2 \ol{\hat f(\rho_2)} + \cos((t+T) \rho_2) \ol{ \hat g(\rho_2)} \right) \\
& \qquad \qquad \cdot \cos (r \rho_1 -\tau)  \cos (r \rho_2-\tau) (\rho_1 \rho_2)^\mu  \, d\rho_1 d\rho_2 \, e^{-\eps r}\, dr.
\end{aligned}
\end{equation}
We now proceed to estimate $I(T,t)$, and then show later that the higher order corrections to the Bessel asymptotics  contribute terms that vanish as~$t\to\infty$. 
To be more precise, we shall show at the end of the proof  that
\begin{equation}\label{E I close}
 \forall\: T,t \ge 0, \qquad   \frac{(2\pi)^d}{2 } \| \partial_t u(t+T) \|_{L^2(|x| \ge t)}^2 =  I(T,t)  + O(t^{-1}) \text{\ \ as\ \ }t\to\infty.
\end{equation}
First, we expand $I$ as follows: With $\mu:=\nu + \frac12=\frac{d-1}{2}$, 
\begin{equation}\label{I int}
\begin{aligned}
\MoveEqLeft I(T,t)  = {\frac{1}{2\pi}} \lim_{\eps\to0+} \int\limits_t^\infty \int\limits_0^\infty \int\limits_0^\infty 
\Big( \sin ((t+T)\rho_1)\sin ((t+T)\rho_2) \rho_1 \hat f(\rho_1) \rho_2 \ol{\hat f(\rho_2)} \\
&  - \sin ((t+T)\rho_1) \cos ((t+T)\rho_2) \rho_1 \hat f(\rho_1) \ol{\hat g(\rho_2) } \\
& - \sin ((t+T)\rho_2) \cos ((t+T)\rho_1) \rho_2 \ol{\hat f(\rho_2)} \hat g(\rho_1) \\
&  + \cos ((t+T)\rho_1) \cos ((t+T)\rho_2) \hat g(\rho_1) \ol{\hat g(\rho_2) } \Big) \\
&  \cdot \big[ (-1)^\nu \sin (r (\rho_1 + \rho_2)) +  \cos (r (\rho_1 - \rho_2))  \big]  (\rho_1 \rho_2)^{\mu}\, d\rho_1 d\rho_2\,
 e^{-\eps r}\, dr.
\end{aligned}\end{equation}
Inserting \eqref{ker1}, \eqref{ker2} into~\eqref{I int} yields 
\[   I(T,t) =   I_1(t+T) +  {(-1)^\nu}  I_2(T,t) -  I_3(T,t), \]
where
\begin{align*}
I_1(s) & = \frac12 \Re \int\limits_0^\infty 
\big( \sin^2 (s\rho) \rho^2 |\hat f(\rho)|^2 - 2 \sin (2s\rho) \rho \hat f(\rho) \ol{ \hat g(\rho) }
 + \cos^2 (s\rho) |\hat g(\rho)|^2\big)\, \rho^{d-1} d\rho \\
I_2(T,t) & =  \frac{1}{4\pi} \int\limits_0^\infty \int\limits_0^\infty
\Big( \big( \cos((t+T)(\rho_1-\rho_2)) - \cos((t+T)(\rho_1+\rho_2)) \big) \rho_1 \hat f(\rho_1) \rho_2 \ol{ \hat f(\rho_2)  }
\\
& \qquad - \big( \sin( (t+T) (\rho_1+\rho_2) ) + \sin ((t+T)(\rho_1-\rho_2)) \big) \rho_1 \hat f(\rho_1) \ol{ \hat g(\rho_2) } \\
& \qquad - \big( \sin( (t+T) (\rho_1+\rho_2) ) - \sin ((t+T)(\rho_1-\rho_2)) \big)  \hat g(\rho_1) \rho_2 \ol{ \hat f(\rho_2) } \\
& \qquad + \big( \cos((t+T)(\rho_1-\rho_2)) + \cos((t+T)(\rho_1+\rho_2)) \big) \hat g(\rho_1) \ol{ \hat g(\rho_2) } \Big) \\
& \qquad \qquad  \cdot   \frac{\cos (t (\rho_1 + \rho_2))}{\rho_1+\rho_2}  (\rho_1 \rho_2)^{\mu} d\rho_1 d\rho_2 \\
I_3(T,t) & = \frac{1}{4\pi} \int\limits_0^\infty \int\limits_0^\infty
\Big( \big( \cos((t+T)(\rho_1-\rho_2)) - \cos((t+T)(\rho_1+\rho_2)) \big) \rho_1 \hat f(\rho_1) \rho_2 \ol{ \hat f(\rho_2) }
\\
& \qquad - \big( \sin( (t+T) (\rho_1+\rho_2) ) + \sin ((t+T)(\rho_1-\rho_2)) \big) \rho_1 \hat f(\rho_1) \ol{ \hat g(\rho_2) } \\
& \qquad - \big( \sin( (t+T) (\rho_1+\rho_2) ) - \sin ((t+T)(\rho_1-\rho_2)) \big)  \hat g(\rho_1) \rho_2 \ol{ \hat f(\rho_2) } \\
& \qquad + \big( \cos((t+T)(\rho_1-\rho_2)) + \cos((t+T)(\rho_1+\rho_2)) \big) \hat g(\rho_1)\ol{ \hat g(\rho_2) } \Big) \\  
& \qquad \qquad \cdot \frac{\sin (t (\rho_1 - \rho_2)) }{\rho_1-\rho_2}  (\rho_1 \rho_2)^{\mu} d\rho_1 d\rho_2.
\end{align*}
Passing to the limit $s\to\infty$ (by Riemann-Lebesgue or using the Schwartz property of the integrand) yields 
\EQ{\nn 
&\ds I_1(s) \tend \frac{1}{4} \int_0^\infty  (\rho^2 |\hat f(\rho)|^2 + |\hat g(\rho)|^2) \rho^{d-1}\,  d\rho =: I_1(\infty)\\
&I_1(\infty) = \frac{(2\pi)^d}{4|\m S^{d-1}|} ( \| f \|^2_{\dot H^1} + \| g \|_{L^2}^2)
}
where the second line uses the Plancherel formula $\|\hat{g}\|_2^2=(2\pi)^d\|g\|_2^2$ in $L^2(\R^d)$. 
Next,  for the the terms containing the Hankel-transform kernel $\frac{1}{\rho_{1}+\rho_{2}}$ we claim that the following representation holds:
\begin{equation}\label{eq:I2claim}
\begin{aligned}
I_2(T,t)
& = \frac{1}{8\pi} \iint \frac{\cos(T(\rho_1+\rho_2))}{\rho_1+\rho_2} \big( - \rho_1 \hat f(\rho_1) \rho_2 \ol{ \hat f(\rho_2) } + 
\hat g(\rho_1) \ol{ \hat g(\rho_2) } \big) (\rho_1 \rho_2)^\mu  \, d\rho_1 d\rho_2 \\
& \qquad  -  \frac{1}{4\pi}\Re\iint \frac{\sin(T(\rho_1+\rho_2))}{\rho_1+\rho_2} \rho_1 \hat f(\rho_1) \ol{ \hat g(\rho_2) }  (\rho_1 \rho_2)^\mu  \, d\rho_1 d\rho_2+ \wt I_2(T,t),
\end{aligned}
\end{equation}
where 
\begin{multline*}
 \forall\: T,t\ge 0, \quad  |\wt I_2(T,t)| \le c(d) (\| f \|_{\dot H^1}^2+ \| g \|_{L^2}^2), \quad
  \forall\: T \ge 0, \quad  \wt  I_2(T,t)  \to 0 \text{ as } t \to +\infty. 
 \end{multline*}
 Moreover, the convergence as $t\to\I$ here holds uniformly in $T\ge0$. 
 
\noindent To verify this claim, notice first that as $\rho_1, \rho_2 \in [\rho_*,\rho^*]$ for some $\rho_*,\rho^* >0$, the denominator $\ds \frac{1}{\rho_1+\rho_2}$ does not create any singularity.
We  simplify  the trigonometric terms as follows, denoting by $z$ either $\sin$ or $\cos$ (which can change from one line to the next):
\begin{align*}
&z ((t+T)(\rho_1-\rho_2)) \cos (t (\rho_1 + \rho_2)) \\& = \frac{1}{2} \big[ z (2t \rho_1 + T(\rho_1-\rho_2)) + z( -2t \rho_2 + T(\rho_1 - \rho_2)) \big].
\end{align*}
Note that there is no complete cancellation of $t$ in this process. Hence for all terms of the type 
\[ I_1(t,T) := \iint z (2 t\rho_1 + T (\rho_1 - \rho_2)) \frac{h_1(\rho_1) h_2(\rho_2)}{\rho_1 + \rho_2}  (\rho_1 \rho_2)^\mu \,   d\rho_1 d\rho_2 \]
(and symmetrically in $\rho_1$ and $\rho_2$), we see that for all $T$ one has  $I_1(t,T) = o_t(1)$ as $t \to +\infty$, and uniformly in $T\ge0$.
The uniformity is established as follows: by the support and smoothness properties of $h_1$, 
\EQ{\nn 
& \iint e^{\pm i\big[2 t\rho_1 + T (\rho_1 - \rho_2)\big]} \, \frac{h_1(\rho_1) h_2(\rho_2)}{\rho_1 + \rho_2}  (\rho_1 \rho_2)^\mu \,   d\rho_1 d\rho_2  \\
&= \iint e^{\pm i (2 t+T)\rho_1} \,   \frac{h_1(\rho_1)}{\rho_1 + \rho_2}  (\rho_1 \rho_2)^\mu \,   d\rho_1\, e^{\mp iT \rho_2 }  h_2(\rho_2) \, d\rho_2 \\
&= \mp \frac{1}{i(2t+T)} \iint e^{\pm i (2 t+T)\rho_1} \,\partial_{\rho_1} \Big( \frac{h_1(\rho_1)\rho_1^\mu  }{\rho_1 + \rho_2} \Big) \,   d\rho_1\: \rho_2^\mu  e^{\mp iT \rho_2 }  h_2(\rho_2)\, d\rho_2
}
as desired. 
For the terms with $z((t+T)(\rho_1+\rho_2))$ 
we use the identity
\[ \sin((t+T)(\rho_1 + \rho_2)) \cos(t(\rho_1 + \rho_2)) = \frac{1}{2} \big[ \sin ((2t+T)(\rho_1 + \rho_2)) + \sin(T(\rho_1+\rho_2)) \big]. \]
The first term yields a contribution of $o_t(1)$ as before whence
\begin{multline*}
\!\!\!\!\iint  \!\! \sin((t+T)(\rho_1+\rho_2)) \Big(\rho_1 \hat f(\rho_1) \ol{  \hat g(\rho_2) } + 
\hat g(\rho_1) \rho_2 \ol{ \hat f(\rho_2) } \Big) \frac{\cos (t (\rho_1 + \rho_2))}{\rho_1+\rho_2}  (\rho_1 \rho_2)^{\mu} d\rho_1 d\rho_2 \\
=  \Re \iint \frac{\sin(T(\rho_1+\rho_2))}{\rho_1+\rho_2} \rho_1 \hat f(\rho_1) \ol{ \hat g(\rho_2)}  (\rho_1 \rho_2)^{\mu} d\rho_1 d\rho_2 + o_t(1)
\end{multline*}
as $t\to\I$, uniformly in $T\ge0$. 
We also used the symmetry here to reduce to one pair of functions. 
In the same way, 
\[ \cos((t+T)(\rho_1 + \rho_2)) \cos(t(\rho_1 + \rho_2)) = \frac{1}{2} \big[ \cos ((2t+T)(\rho_1 + \rho_2)) + \cos(T(\rho_1+\rho_2)) \big].
 \]
The first term makes a contribution of $o_t(1)$, again uniformly in $T\ge0$,  and thus
\begin{multline*}
\!\!\!\!\iint  \!\! \cos((t+T)(\rho_1+\rho_2)) \big) \Big(-\rho_1 \hat f(\rho_1) \rho_2 \ol{ \hat f(\rho_2) } + \hat g(\rho_1) \ol{ \hat g(\rho_2) } \Big) 
\frac{\cos (t (\rho_1 + \rho_2))}{\rho_1+\rho_2}  (\rho_1 \rho_2)^{\mu} \, d\rho_1 d\rho_2 \\
= \frac{1}{2} \iint \frac{\cos (T(\rho_1+\rho_2))}{\rho_1+\rho_2} \Big( -\rho_1 \hat f(\rho_1) \rho_2 \ol{ \hat f(\rho_2)} + \hat g(\rho_1) 
\ol{\hat g(\rho_2) }\Big)  (\rho_1 \rho_2)^{\mu} \, d\rho_1 d\rho_2 + o_t(1) 
\end{multline*}
and we have proved \eqref{eq:I2claim}. 

 It remains to deal with the $I_3$-term. Let  $\hat{h}_1, \hat{ h}_2$ denote any of the functions $$\m 1_{[0,+\infty)}(\rho) \rho^{\mu+1} \hat f(\rho) \text{\ \ or\ \ } \m 1_{[0,+\infty)}(\rho) \rho^{\mu} {\hat g}(\rho).$$
Here $\hat{h}_j$ are the {\em one-dimensional} Fourier transforms. 
We write the trigonometric factors in exponential form: all the terms are of the type
\EQ{\label{shift}
\begin{aligned}
\MoveEqLeft \iint e^{i(t+T)(\rho_1 \pm \rho_2)} \frac{\sin(t(\rho_1-\rho_2))}{\rho_1-\rho_2} \hat h_1(\rho_1) \ol{\hat h_2 (\rho_2) } \,  d\rho_1 d\rho_2 \\
& = \frac12\int (\widehat{\m 1_{[-t,t]}} * (e^{i(t+T)\rho_1} \hat h_1))(\rho_2) \, \ol{ e^{\mp i (t+T) \rho_2} \hat h_2(\rho_2)}\,  d\rho_2 \\
& = \frac{1}{4\pi} \int_{-t}^t h_1(r + (t+T)) \ol{ h_2(r \mp (t+T)) }\,  dr
\end{aligned}
}
where we used Plancherel on the last line. Via Cauchy-Schwarz we can bound these terms by
\EQ{\label{h1h2}
 \| h_1 \|_{L^2(|x| \ge T)} \| h_2 \|_{L^2(|x| \ge T)}.
 }
Due to the distinction between the Fourier transform on the line and in~$\R^d$ we cannot simply express
the previous expression by one involving the energy of~$(f,g)$ over~$\{|x|>T\}$.  However, it is clear that~\eqref{h1h2}
can be made arbitrarily small by taking $T\ge T_*$.  

In summary, we arrive at the following preliminary conclusion:

{\em Given $\eps>0$ there exists $T_*=T_*(\eps,f,g)$ such that the following holds: for  any $T \ge T_*$   
\begin{equation} 
\label{I pieces}
\begin{aligned}
&  I(T,t) \ge I_1(\infty)(1-\eps)  \\
& \; + \frac{(-1)^\nu}{8\pi}   \iint \frac{\cos(T(\rho_1+\rho_2))}{\rho_1+\rho_2} \big( - \rho_1 \hat f(\rho_1) \rho_2 \ol{\hat f(\rho_2)} + \hat g(\rho_1) \ol{\hat g(\rho_2)} \big) (\rho_1 \rho_2)^{\mu}\,  d\rho_1 d\rho_2 \\
&\;  - \frac{(-1)^\nu}{4\pi}   \Re \iint \frac{\sin(T(\rho_1+\rho_2))}{\rho_1+\rho_2} \rho_1 \hat f(\rho_1) \ol{\hat g(\rho_2) } (\rho_1 \rho_2)^{\mu} \,  d\rho_1 d\rho_2+ \wt I(T,t),
\end{aligned}
\end{equation}
where 
\[ \forall\: t\ge 0, \quad  |\wt I(T,t)| \le c_1(d) (\| f \|_{\dot H^1}^2+ \| g \|_{L^2}^2), \quad \text{and} \quad  \wt I(T,t) \to 0 \text{ as } t \to +\infty. \]
 The constants here do not depend on $T$, and the vanishing of $\tilde I$ as $t\to\I$ holds uniformly in $T\ge0$. }

To proceed we first note that 
\begin{equation} \label{low_sinc}
\frac{1}{T} \int_0^T \big( \| f \|_{\dot H^1(|x| \ge \tau)}^2 + \| g \|_{L^2(|x| \ge \tau)}^2 \big)\,   d\tau \tend 0
\end{equation}
as $T\to\I$. 
The double integrals in \eqref{I pieces} will be dealt with by randomizing $T$, in other words, by taking averages in~$T$. This process becomes degenerate for small frequencies
$\rho_1, \rho_2$. However,  by the uncertainty principle (which amounts to an application of Bernstein's inequality), 
these small frequencies occur only with small probability and  can therefore be ignored. 

To be specific, we rely on the following simple fact:  
let $h \in L^2(\R^d)$ be such that $\ds \| h \|_{L^2(|x| \ge R)} \le \delta \| h \|_{L^2}$. Then, with $\hat{h}$ being the Fourier transform in~$\R^d$, 
\begin{equation}\label{UP}
 \| \hat h \|_{L^2(|\xi| \le \rho)} \le c(d) ((R\rho)^{\frac{d}{2}} + \delta) \| h \|_{L^2}.
 \end{equation}
To prove this property, let $h_1 := h \m 1_{[|x| \le R]}$, $h_2 := h -h_1$. Then $\| \hat h_2 \|_{L^2} \le c(d) \delta \| h \|_{L^2}$ and
\[ \| \hat h_1 \|_{L^\infty} \le \| h_1 \|_{L^1} \le c(d)R^{\f{d}{2}} \| h_1 \|_{L^2} \le c(d)R^{\f{d}{2}} \| h \|_{L^2}. \]
Now, by Cauchy-Schwarz,
\[ \| \hat h_1 \|_{L^2(|\xi| \le \rho)} \le \sqrt{| \{ |\xi| \le \rho \}|}\,  \| \hat h_1 \|_{L^\infty} \le c(d) (R\rho)^{\frac{d}{2}} \| h \|_{L^2}. \]
As $\hat h = \hat h_1 + \hat h_2$, \eqref{UP} follows.

We apply \eqref{UP} to establish the following ``randomized estimate'' on the double integrals in~\eqref{I pieces}. We formulate it as a general principle:

{\em 
Given $\delta > 0$ and any $h_1,h_2 \in L^2(\R^d)$ radial, there exists $T^*= T^*(\delta,h_1,h_2)$ such that for all $T \ge T^*$,}
\begin{equation}\label{AV}
\left| \frac{1}{T} \int_0^T \iint \frac{e^{i\tau(\rho_1+\rho_2)}}{\rho_1 +\rho_2}  \hat h_1(\rho_1) \ol{ \hat h_2(\rho_2)}\,  (\rho_1\rho_2)^\mu\, d\rho_1 d\rho_2 d\tau \right|  \le c(d) \delta^2 \| h_1 \|_{L^2} \| h_2 \|_{L^2}.
\end{equation}
With $T, \rho > 0$ to be determined later, we split the integral into two parts:
\begin{align*}
I_{\le \rho} (t,T) &:=  \frac{1}{T} \int_0^T \iint_{\rho_1 + \rho_2 \le \rho} \frac{e^{i\tau(\rho_1+\rho_2)}}{\rho_1 +\rho_2} \hat h_1(\rho_1) \ol{ \hat h_2(\rho_2)}\, (\rho_1\rho_2)^\mu\, d\rho_1 d\rho_2 d\tau,   \\
I_{\ge \rho} (t,T) &:=  \frac{1}{T} \int_0^T \iint_{\rho_1 + \rho_2 \ge \rho} \frac{e^{i\tau(\rho_1+\rho_2)}}{\rho_1 +\rho_2}  \hat h_1(\rho_1) \ol{ \hat h_2(\rho_2)}\, (\rho_1\rho_2)^\mu\,  d\rho_1 d\rho_2 d\tau
\end{align*}
where it is understood that $\rho_1,\rho_2>0$. 
Then with $R$ as in~\eqref{UP}
\begin{align*}
|I_{\le \rho} (t,T)| & \le \frac{1}{T} \int_0^T \iint \frac{|\hat h_1(\rho_1)| \m 1_{[\rho_1 \le \rho]}  |\hat h_2(\rho_2)| \m 1_{[\rho_2 \le \rho]}}{\rho_1+\rho_2} \, (\rho_1 \rho_2)^{\mu} \, d\rho_1 d\rho_2 d\tau\\
& \le \| H ( |\hat h_1(\rho_1)|\rho_1^\mu \m 1_{[\rho_1 \le \rho]}) \|_{L^2} \| \hat h_2(\rho_2)\rho_2^\mu  \|_{L^2(|\rho_2| \le \rho)} \\
& \le c(d)^2 ((R\rho)^{\frac{d}{2}} + \delta)^2 \| h_1 \|_{L^2} \| h_2 \|_{L^2}.
\end{align*}
where we used $L^2$-boundedness  of the Hankel transform $\ds (Hf)(r) := \int_0^\infty \frac{f(s)}{r+s}\, ds$ and \eqref{UP} to pass to the final estimate. 
For the second term, we  integrate first in $\tau$ 
\begin{align*}
|I_{\ge \rho}(t,T)| & \le \frac{2}{T} \iint_{\rho_1 + \rho_2 \ge \rho} \frac{|\hat h_1(\rho_1) \hat h_2(\rho_2)|}{(\rho_1+\rho_2)^2} \, (\rho_1\rho_2)^\mu\,  d\rho_1 d\rho_2   \\
& \le \frac{2}{\rho T} \langle H(|\hat{h}_1|\rho_1^\mu) , |\hat{h}_2|\rho_2^\mu\,\rangle \le \frac{C(d)}{\rho T} \| h_1 \|_{L^2} \| h_2 \|_{L^2}
\end{align*}
where $\langle \cdot,\cdot\rangle$ is the $L^2(0,\infty)$-pairing. Taking first $R$ large (depending on~$\delta, h_1,h_2$), then $\rho$ small, and
finally~$T$ large implies~\eqref{AV}.

It is now a simple matter to finish the estimation of the principal term. 
Indeed, fix a small $\eps>0$ (to be determined later) and let $T^*, T_* \ge 0$ be sufficiently large. 
Then for all $T \ge \max(T^*,T_*)$ and $t \ge 0$, we obtain the following lower bound on~\eqref{I pieces}: 
\[ \frac{1}{T} \int_0^{T} I(\tau,t) \, d\tau \ge (1 - \eps) I_1(\infty) - \frac{1}{T} \int_0^T |\wt I(\tau,t)| \, d\tau . \]
By the asymptotic behavior of $\wt I(\tau,t)$ we see that given $\eps>0$  there exists $T_{0}$, depending on~$\eps, f ,g $ and~$d$, such that for all $T\ge T_0$ 
\[
\limsup_{t\to\I} \frac{1}{T} \int_0^{T} I(\tau,t) \, d\tau \ge (1- \eps)  I_1(\infty).
\]
Recall that so far we have only dealt with the kinetic part of the energy, i.e., the one given by $\partial_{t}u$. 
Note that this gives us only {\em half } of what we need, since $I_1(\infty)$ equals half of the full energy. The other half
 comes from~$\partial_r u(t,r)$, the contribution of which is given by   
\begin{equation}
\nn
\begin{aligned} 
\MoveEqLeft 
(2\pi)^d |\m S^{d-1}|^{-1}  \frac12  \| \partial_r u(t+T) \|_{L^2(|x| \ge t)}^2  = (2\pi)^d\int_{t}^\infty \frac12 |\partial_r u(t+T,r)|^2 \, r^{d-1} \, dr \\
&= (2\pi)^d \lim_{\eps\to0+} \int_{t}^\infty \frac12 |\partial_r u(t+T,r)|^2  r^{d-1} e^{-\eps r} \, dr \\
& = \lim_{\eps\to0+}\int_t^\infty \iint  \frac12 \left(\cos ((t+T)\rho_1) \rho_1 \hat f(\rho_1) + \sin((t+T) \rho_1) \hat g(\rho_1) \right) \\
& \qquad \qquad \cdot \left(\cos ((t+T)\rho_2) \rho_2 \ol{ \hat f(\rho_2)} + \sin((t+T) \rho_2) \ol{\hat g(\rho_2)} \right) \\
& \qquad \qquad \cdot  J'_{\nu}(r \rho_1)  J'_{\nu}(r \rho_2) (r^2 \rho_1 \rho_2)^{-\nu}  (\rho_1 \rho_2)^{d-1} \, d\rho_1 d\rho_2\,  r^{d-1}\, e^{-\eps r} \, dr + o_t(1)
\end{aligned}
\end{equation}
as $t\to\infty$, see \eqref{u_r}.  The final term here results from the derivatives in $r$ falling on the $r^{-2\nu}$ weight outside of the Bessel functions, see below
for the treatment of such error terms.  
Plugging in the asymptotics from~\eqref{eq:Jnu exp}, and performing the same type of arguments as before  
now yields  
\[
\limsup_{t\to\I} \frac{1}{T} \int_0^{T} \| \nabla_{t,x} u(t) \|_{L^2(|x|>t-\tau)}^2 \, d\tau \ge (1 - \eps) \| (f,g) \|_{\dot H^1 \times L^2}^2
\]
for all $T\ge T_{0}$. We also used the Plancherel identity $\|\hat{f}\|_2^2=(2\pi)^d \|f\|_2^2$. 
By the monotonicity of the exterior energy, we can take $T=T_{0}$ which leads to the desired result. 

It remains to verify the dominance of the leading order terms of the Bessel expansion as expressed by~\eqref{E I close}. 
This is very similar to the corresponding argument in the proof of Theorem~\ref{th1}. Indeed, subtracting~\eqref{Idef2} from~\eqref{Ekinetic2}  yields, with $\omega_j$ as in~\eqref{est:w}, 
\begin{align*}
I(T,t) & : = {\frac{2}{\pi}} \lim_{\eps\to0+}\int\limits_t^\infty \int\limits_0^\infty \int\limits_0^\infty 
\left(-\sin ((t+T)\rho_1) \rho_1 \hat f(\rho_1) + \cos((t+T) \rho_1) \hat g(\rho_1) \right) \\
& \qquad \qquad \cdot \left(-\sin ((t+T)\rho_2) \rho_2 \ol{\hat f(\rho_2)} + \cos((t+T) \rho_2) \ol{ \hat g(\rho_2)} \right) \\
& \qquad \qquad\cdot \big[ \big(\omega_2(r\rho_1) + \omega_2(r\rho_2) + \omega_2(r\rho_1) \omega_2(r\rho_2) \big)   \cos (r \rho_1 -\tau)  \cos (r \rho_2-\tau) \\
&\qquad\qquad\quad  + \omega_1(r\rho_1) (1+\omega_2(r\rho_2))
\sin (r \rho_1 -\tau)  \cos (r \rho_2-\tau) \\
& \qquad\qquad\quad + \omega_1(r\rho_2) (1+\omega_2(r\rho_1))
\sin (r \rho_2 -\tau)  \cos (r \rho_1-\tau) \\
&\qquad\qquad \quad + \omega_1(r\rho_1) \omega_1(r\rho_2) \sin(r\rho_1-\tau)\sin(r\rho_2-\tau)\big]  (\rho_1 \rho_2)^\mu  \, d\rho_1 d\rho_2 \, e^{-\eps r}\, dr.
\end{align*}
All terms here are treated in a similar fashion. As one example, consider for all $\eps>0$
the error term
\begin{multline*}
E_1(\eps):= \int\limits_t^\infty \int\limits_0^\infty \int\limits_0^\infty \sin ((t+T) \rho_1) \sin ((t+T) \rho_2) \cos(r\rho_1 - \tau) \sin (r\rho_2-\tau) \w_1(r\rho_2) \\
\cdot \hat f(\rho_1) \overline{\hat f(\rho_2)} (\rho_1 \rho_2)^{\mu+1} e^{-\eps r} \, d\rho_1 d\rho_2 dr,
\end{multline*}
As before, we write
\[ \cos(r\rho_1 - \tau) \sin (r\rho_2-\tau) = -\frac{1}{2} \big[ (-1)^\nu \cos(r (\rho_1+\rho_2)) + \sin(r(\rho_1- \rho_2)) \big], \]
expand the trigonometric functions on the right-hand side into complex exponentials, 
and perform an integration by parts in the $r$ variable as in~\eqref{IBP}. We apply this with $\sigma=\rho_1+\rho_2$ and $\sigma=\rho_1-\rho_2$ to the fully expanded form of $E_1(\eps)$ as explained above. 
In both cases one has the uniform bounds
\[
\sup_{\eps>0}\Big\| \int_{-\infty}^\infty \frac{\phi(\rho_2)}{(\rho_1 \pm \rho_2) \pm i\eps} \, d\rho_2\Big\|_{L^2(\rho_1)} \le C\|\phi\|_2.
\]
In order to use this, we distribute the exponential factors as well as all  weights over the functions $\hat f(\rho_1)$ and $ \overline{\hat f(\rho_2)} $, respectively.
For the first term on the right-hand side of~\eqref{IBP} we then obtain an estimate $O(t^{-1})$ from the decay of the weight~$\omega$, whereas for the integral in~\eqref{IBP}
we obtain a $O(r^{-2})$-bound via 
\[
\sup_{\rho>0} |\omega'(r\rho)\rho^2| \le C\, r^{-2}
\]
which then leads to the final bound
\[
\int_t^\infty O(r^{-2})\, dr = O(t^{-1}).
\]
The $O$-here are   uniform in~$\eps>0$.  Note that various $\rho$-factors which are introduced by the $\omega$-weights  are
harmless due to our standing assumption that $0<\rho_*<\rho <\rho^*$. 
\end{proof}

\begin{proof}[Proof of Theorem~\ref{thm6}]
This is an immediate consequence of Proposition~\ref{prop:inner part 0} and the monotonicity of the energy on the region~$\{|x|\ge t+T\}$. 
\end{proof}

\section{Concentration compactness decompositions} 
\label{bilin}

This final section collects several admittedly technical results which are, however, of crucial importance 
in the implementation of  nonlinear arguments in our followup work~\cite{CKLS1, CKLS2}. 
The main results in this section are as follows:

\begin{itemize}
\item Localization of the energy to the exterior of balls centered at the origin does not affect the Pythagoras theorem for 
the energy in the linear radial concentration-compactness decomposition. This is formulated precisely in
Corolloray~\ref{energy_expansion_cutoff}   below. 
\item Suppose a sequence of radial free waves, uniformly bounded in energy, converges to zero in the Strichartz
sense. Then it will continue to do so if the data are truncated smoothly to (the exterior of) balls centered at the origin,
but of arbitrary radii. See Lemma~\ref{lem:tl2} below. 
\end{itemize}

The first fact is established in~\cite{DKM13}, and the second one in~\cite{DKM11}, both in  three dimensions.
By means of the machinery
of the previous sections we can extend their validity to even dimensions.

\subsection{A bilinear convergence property}

The main technical issue in the proof of Corolloray~\ref{energy_expansion_cutoff} is addressed in the following
bilinear result.   Note the inclusion of the cut-offs $\{|x|>r_n\}$ or $\{|x|<r_n\}$   in~\eqref{eq:Uw}, \eqref{eq:Uw2}.

\begin{lem}
\label{lem:bilinear}
Let $\vec w_n=(w_{n,0},w_{n,1})$ be a bounded sequence in of radial functions in $\dot H^1 \times L^2$. Let $t_n, r_n$ be two sequences ($r_n \ge 0$).
Assume that $\nabla_{x,t} S(-t_n) \vec w_n \weak 0$ in $L^2$ as $n \to \infty$. Then
for any $\vec U =(u_0,u_1) \in (\dot H^1\times L^2)(\R^d)$,  one has 
\begin{gather}
\label{eq:Uw}  
\int_{|x|> r_n} \nabla_{x,t} S(t_n) \vec U \cdot (\nabla_x w_{n,0}, w_{n,1}) \, dx \to 0 \quad \text{as} \quad n \to +\infty, \\
\int_{|x| < r_n} \nabla_{x,t} S(t_n) \vec U \cdot (\nabla_x w_{n,0}, w_{n,1}) \, dx \to 0 \quad \text{as} \quad n \to +\infty \label{eq:Uw2}
\end{gather}
\end{lem}

\begin{proof} 
By conservation of the linear energy, one has 
\EQ{& \int_{\m R^d} \nabla_{x,t} S(t_n) \vec U \cdot (\nabla_x w_{n,0}, w_{n,1}) \, dx \\
&= \int_{\m R^d} \nabla_{x,t} \vec U \cdot S(-t_n)(\nabla_x w_{n,0}, w_{n,1}) \, dx \to 0 \quad \text{as} \quad n \to +\infty. 
}
Hence, \eqref{eq:Uw2} and \eqref{eq:Uw} are equivalent.  

By unitarity of the evolution we may assume that $u$ is a Schwartz function with Fourier support away from the origin. Also it suffices to show the claim assuming that the sequences
\[ (t_n)_n, \quad (r_n)_n, \quad (t_n - r_n)_n \quad  \text{and} \quad (t_n+r_n)_n \quad \text{have a limit in } \overline{\m R}. \]
If $t_n$ has a finite limit, then $S(t_n) \vec U$ converges strongly in $L^2$ and $(\nabla_x w_{n,0}, w_{n,1})$ converges weakly in $L^2$. Now recall the following simple fact: if $f_n \weak f$ weakly in $L^2$, and $\alpha_n \to \alpha \in \overline{\m R}$, the dominated convergence theorem shows that
\[ \m 1_{|x| \ge \alpha_n}  f_n \weak \m 1_{(\alpha,+\infty)} f \quad \text{weakly in } L^2. \]
Applying this to $\alpha_n=r_n$ and $f_n = (\nabla_x w_{n,0}, w_{n,1})$ yields the result in this case. 

We now turn to the case when $\lim t_n \in \{ \pm \infty \}$. We have shown above that the sequence $\nabla_{x,t} S(t_n) \vec U$ asymptotically concentrates its $L^2$ mass where $||x|-|t_n|| \le R$. In particular,
\[ \int_{|x| \le |t_n|/2} |\nabla_{x,t} S(t_n) \vec U|^2 \to 0. \]
If $r_n$ is bounded, it then transpires  that 
\[ \int_{|x|\le r_n} \nabla_{x,t} S(t_n) \vec U \cdot (\nabla_x w_{n,0}, w_{n,1}) \, dx \to 0, \]
and we are done with this case. 

It remains to treat the case where both $(t_n)_n$ and $(r_n)_n$ have infinite limits.
We proceed as in the proof of Proposition \ref{prop:inner part 0}, using the Fourier representation and the Bessel functions $J_\nu$ with $\ds \nu=\frac{d-2}{2}$. 
Retaining only the leading orders in the expansions of these functions the dominant contribution to~\eqref{eq:Uw} is given by
\begin{align*}
\MoveEqLeft \int_{r_n}^\infty\int_0^\infty \big( \cos(t_n \rho)\rho\wh{u_0}(\rho)+\sin(t_n\rho)\wh{u_1}(\rho)\big)\sin(r\rho-\tau)(r\rho)^{-\nu-\frac12}\rho^{d-1}\,d\rho  \\
&\qquad \int_0^\infty \wh{w_{n,0}}(\sigma)\sigma\sin(r\sigma-\tau)(r\sigma)^{-\nu-\frac12} \sigma^{d-1}\,d\sigma\, e^{-\eps r} r^{d-1}\, dr \\
\MoveEqLeft + \int_{r_n}^\infty \int_0^\infty \big( -\sin(t_n\rho) \rho\wh{u_0}(\rho)+\cos(t_n\rho)\wh{u_1}(\rho)\big) \cos(r\rho-\tau)(r\rho)^{-\nu-\frac12}\rho^{d-1}\,d\rho  \\
& \qquad \int_0^\infty \wh{w_{n,1}}(\sigma)\cos(r\sigma-\tau)(r\sigma)^{-\nu-\frac12} \sigma^{d-1}\,d\sigma\, e^{-\eps r} r^{d-1}\, dr
\end{align*}
in the limit $\eps\to0+$. Carrying out the $r$-integration and passing to the limit yields the expression
\begin{equation}\label{eq:limit}
\begin{aligned}
\MoveEqLeft \int_0^\infty\int_0^\infty \big( \cos(t_n \rho)\rho\wh{u_0}(\rho)+\sin(t_n\rho)\wh{u_1}(\rho)\big) \rho^{\frac{d-1}{2}} \, \sigma \wh{w_{n,0}}(\sigma) \sigma^{\frac{d-1}{2}} \\
&\qquad \Big( \pi \delta_0(\rho-\sigma) - \frac{\sin(r_n (\rho-\sigma))}{\rho-\sigma} - (-1)^\nu \frac{\cos(r_n(\rho+\sigma))}{\rho+\sigma}\Big) 
\,d\rho d\sigma  \\
\MoveEqLeft + \int_0^\infty\int_0^\infty \big( - \sin(t_n \rho)\rho\wh{u_0}(\rho)+\cos(t_n\rho)\wh{u_1}(\rho)\big) \rho^{\frac{d-1}{2}} \,   \wh{w_{n,1}}(\sigma) \sigma^{\frac{d-1}{2}} \\
&\qquad \Big( \pi \delta_0(\rho-\sigma) - \frac{\sin(r_n(\rho-\sigma))}{\rho-\sigma} + (-1)^\nu \frac{\cos(r_n(\rho+\sigma))}{\rho+\sigma}\Big) 
\,d\rho d\sigma 
\end{aligned}   
\end{equation} 
The $\delta_0$ make the following contribution to \eqref{eq:limit}:
\begin{multline*}
\int_0^\infty \rho\wh{u_0}(\rho) \big[ \cos(t_n\rho) \rho \wh{w_{n,0}}(\rho)   -\sin(t_n\rho)  \wh{w_{n,1}}(\rho)  \big] \rho^{d-1}\, d\rho \\
+ \int_0^\infty \wh{u_1}(\rho) \big[ \sin(t_n\rho) \rho \wh{w_{n,0}}(\rho)   + \cos(t_n\rho)  \wh{w_{n,1}}(\rho)  \big] \rho^{d-1}\, d\rho
\end{multline*}
which tends to $0$ by the assumption on $w_n$.  Next, we extract the terms involving the Hilbert transform kernel from~\eqref{eq:limit} (ignoring multiplicative constants):
\begin{equation}
\label{eq:Hilbert}
\begin{aligned}
&\int_0^\infty\int_0^\infty \Big\{   \rho\wh{u_0}(\rho)  \big[  \cos(t_n\rho) \sigma \wh{w_{n,0}}(\sigma)   -\sin(t_n\rho)  \wh{w_{n,1}}(\sigma)  \big] + \\
& \qquad + \wh{u_1}(\rho)  \big[  \sin(t_n\rho) \sigma \wh{w_{n,0}}(\sigma)   + \cos(t_n\rho)  \wh{w_{n,1}}(\sigma)  \big]  \Big\} \frac{\sin(r_n(\rho-\sigma))}{\rho-\sigma}\, (\rho\sigma)^{\frac{d-1}{2}}\, d\rho d\sigma 
\end{aligned}\end{equation}
Using simple trigonometry, the terms involving $u_0$ can be transformed into the following expression:
\begin{align}
\MoveEqLeft \int_0^\infty\!\!\! \int_0^\infty   \!\! \rho\wh{u_0}(\rho) \Big\{ \!\big[ \sin((t_n + r_n)(\rho-\sigma)) - \sin((t_n-r_n)(\rho-\sigma))\big] \cos(t_n \sigma) \sigma \wh{w_{n,0}}(\sigma)  \nn \\
& \!\!\!\! +  \big[ \cos((t_n + r_n)(\rho-\sigma)) - \cos((r_n-t_n)(\rho-\sigma))\big] \sin(t_n \sigma) \sigma \wh{w_{n,0}}(\sigma)  \nn  \\
&  \!\!\!\!-  \big[ \cos((t_n - r_n)(\rho-\sigma)) - \cos((t_n + r_n)(\rho-\sigma))\big] \cos(t_n \sigma)  \wh{w_{n,1}}(\sigma)  \nn  \\
& \!\!\!\! \!\!\!\! -  \big[ - \sin((t_n - r_n)(\rho-\sigma)) + \sin((t_n+r_n)(\rho-\sigma))\big] \sin(t_n \sigma)  \wh{w_{n,1}}(\sigma) \Big\}  \frac{(\rho\sigma)^{\frac{d-1}{2}}}{\rho-\sigma} \, d\rho d\sigma   \nn  \\
& = \begin{aligned}[t] 
& \int_0^\infty \!\!\! \int_0^\infty \rho\wh{u_0}(\rho) \Big\{ \big[ \sin((t_n+r_n)(\rho-\sigma)) + \sin((r_n - t_n)(\rho-\sigma))\big] \\
&  \qquad  \qquad \qquad \times \big( \cos(t_n \sigma) \sigma \wh{w_{n,0}}(\sigma)-\sin(t_n \sigma ) \wh{w_{n,1}}(\sigma) \big)  \\
& +  \big[ \cos((t_n+r_n)(\rho-\sigma)) - \cos((r_n-t_n)(\rho-\sigma))\big] \\
&  \qquad \qquad \qquad \times \big( \sin(t_n \sigma) \sigma \wh{w_{n,0}}(\sigma)  + \cos(t_n \sigma)  \wh{w_{n,1}}(\sigma) \big)  \Big\}  \frac{(\rho\sigma)^{\frac{d-1}{2}}}{\rho-\sigma} \, d\rho d\sigma 
\end{aligned}
\label{eq:u0} 
\end{align}
Define, with $\calF_1$ the Fourier transform on $\R$, 
\[
\tilde u_0 := \calF_1^{-1}\big(\m 1_{\R^+} \rho \wh{u_0}(\rho)\rho^{\frac{d-1}{2}}\big)  \in L^2(\R)
\]
Then with some constant $c$, 
\[
\int_0^\infty e^{\pm iB_n(\rho-\sigma)} \frac{\rho \wh{u_0}(\rho)}{\rho-\sigma}\, \rho^{\frac{d-1}{2}} \,d\rho = c\calF_1(\sign(\cdot\pm B_n)\tilde u_0\big)(\sigma)
\]
If $B_n$ has a limit in $\m R$ or $\pm \infty$, then this converges strongly in $L^2(\R)$: in our case, $B_n$ is $t_n+r_n$ or $t_n-r_n$. Thus, \eqref{eq:u0} can be reduced to the form $\lan v_n, \tilde v_n\ran\to0$ where $v_n$ converges
strongly in~$L^2$ and $\tilde v_n\weak0$ weakly in $L^2$ as $n\to\infty$.   Analogously, the terms involving $u_1$ in~\eqref{eq:Hilbert} are reduced to the following expressions: 
\begin{align*}
\MoveEqLeft \int_0^\infty\int_0^\infty \wh{u_1}(\rho) \Big\{ \big[ \sin((t_n+r_n)(\rho-\sigma)) - \sin((t_n-r_n)(\rho-\sigma))\big] \\
& \qquad \qquad \qquad \qquad \qquad \qquad \big(\sin(t_n \sigma) \sigma \wh{w_{n,0}}(\sigma)  + \cos(t_n \sigma)  \wh{w_{n,1}}(\sigma) \big)   \\
& \qquad \qquad -  \big[ \cos((t_n+r_n)(\rho-\sigma)) - \cos((t_n-r_n)(\rho-\sigma))\big] \\
& \qquad \qquad \qquad \qquad  \big(\cos(t_n \sigma) \sigma \wh{w_{n,0}}(\sigma)-\sin(t_n \sigma)  \wh{w_{n,1}}(\sigma) \big)   \Big\}  \frac{(\rho\sigma)^{\frac{d-1}{2}}}{\rho-\sigma} \, d\rho d\sigma 
\end{align*}
which converges to zero by the same reason. 

It remains to handle the terms in \eqref{eq:limit}  involving the Hankel kernel $\ds \frac{1}{\rho+\sigma}$.   Using the same type of trigonometric identities as above
the terms involving the Hankel kernel as well as  $u_0$ are transformed into the following ones:
\begin{align*}
\MoveEqLeft \int_0^\infty\int_0^\infty \rho \wh{u_0}(\rho) \Big\{ \big[ \sin((t_n+r_n)(\rho+\sigma)) + \sin((t_n-r_n)(\rho+\sigma))\big] \\
& \qquad \qquad \qquad \qquad    (\sin(t_n \sigma) \sigma \wh{w_{n,0}}(\sigma)  + \cos(t_n \sigma)  \wh{w_{n,1}}(\sigma))   \\
& \qquad \qquad +  \big[ \cos((t_n+r_n)(\rho+\sigma)) + \cos((t_n-r_n)(\rho+\sigma))\big] \\
& \qquad \qquad \qquad \qquad   (\cos(t_n \sigma) \sigma \wh{w_{n,0}}(\sigma)-\sin(t_n \sigma)  \wh{w_{n,1}}(\sigma) )   \Big\}  \frac{(\rho\sigma)^{\frac{d-1}{2}}}{\rho+\sigma} \, d\rho d\sigma 
\end{align*}
%In view of our assumption  on $u_0$, integrating by parts in $\rho$ shows that the terms involving $t_n+r_n$ are $O((t_n+r_n)^{-1})$ in $L^2(\R)$ as $n\to\infty$. 
%Therefore, we may again pass to the limit $n\to\infty$ using the weak convergence of the terms involving $w_{n,0}, w_{n,1}$. For the terms involving $t_n-r_n$ 
We proceed as in the case of the Hilbert transform, considering $$\breve u_0 := \calF_1^{-1}\big(\m 1_{\R^-} \rho \wh{u_0}(\rho)\rho^{\frac{d-1}{2}}\big)$$ instead of $\tilde u_0$, and noticing
\[ 
\int_0^\infty e^{\pm iB_n(\rho+\sigma)} \frac{\rho \wh{u_0}(\rho)}{\rho+\sigma}\, \rho^{\frac{d-1}{2}} \,d\rho = c\calF_1(\sign(\cdot\mp B_n)\breve u_0\big)(\sigma).
\]

We argue analogously for the terms involving the Hankel kernel as well as $u_1$, which are of the form
\begin{align*}
\MoveEqLeft \int_0^\infty\int_0^\infty  \wh{u_1}(\rho) \Big\{ \big[ \sin((t_n+r_n)(\rho+\sigma)) + \sin((t_n-r_n)(\rho+\sigma))\big] \\
& \qquad \qquad \qquad \qquad  (\cos(t_n \sigma) \sigma \wh{w_{n,0}}(\sigma)  - \sin(t_n \sigma)  \wh{w_{n,1}}(\sigma))   \\
& \qquad \qquad -  \big[ \cos((t_n+r_n)(\rho+\sigma)) + \cos((t_n-r_n)(\rho+\sigma))\big] \\
& \qquad \qquad\qquad \qquad  (\sin(t_n \sigma) \sigma \wh{w_{n,0}}(\sigma)+\cos(t_n \sigma)  \wh{w_{n,1}}(\sigma) )   \Big\}  \frac{(\rho\sigma)^{\frac{d-1}{2}}}{\rho+\sigma} \, d\rho d\sigma 
\end{align*}
By inspection, these also vanish in the limit $n\to\infty$. 

It remains to deal with the errors resulting from   the lower orders in~\eqref{eq:Jnu exp}. In contrast to the leading order, no use is going
to be made of the weak convergence assumption on~$w_n$. Indeed, just by means of $L^2$-estimation and the gain of (at least) one power stemming
from the $\omega_j$ and $\tilde\omega_j$ factors in~\eqref{eq:Jnu exp}, one obtains a $O(r_n^{-1})$ bound on all of the contributions of these terms
to the left-hand side of~\eqref{eq:Uw} (recall our assumption $\rho>\rho_*>0$, and the same for $\sigma$). 
To be more specific, the error terms are of the form
\EQ{\label{U0w1}
\int_{r_n}^\I \big( U_{n,0}(r) w_{n,1}'(r)  + U_{n,1}(r) w_{n,0}'(r) + U_{n,1}(r) w_{n,1}'(r)  \big)  r^{d-1}\, dr
}
where
\EQ{ \nn
U_{n,1}(r) &= \int_0^\I \Big(\cos(t_n\rho)\wh{u_0}(\rho) + \frac{\sin(t_n\rho)}{\rho} \wh{u_1}(\rho)\Big) \big( \tilde \omega_1(r\rho)\cos(r\rho-\tau)\\
&\qquad  - \tilde\omega_2(r\rho)\sin(r\rho-\tau)\big) (r\rho)^{-\nu-\frac12} \rho^{d}\, d\rho \\
w_{n,1}(r) &=\int_0^\I \wh{w_n}(\sigma) \big( \tilde \omega_1(r\sigma)\cos(r\sigma-\tau) - \tilde\omega_2(r\sigma)\sin(r\sigma-\tau)\big)  (r\sigma)^{-\nu-\frac12}\sigma^d\, d\sigma
}
Let us consider the first term in~\eqref{U0w1}:
\EQ{\label{rn error}
& \int_{r_n}^\I U_{n,0}(r) w_{n,1}'(r) r^{d-1}\, dr\\
&= \lim_{\eps\to0+} \int_{t_n+A}^\I \int_0^\I \Big(\cos(t_n\rho)\wh{u_0}(\rho) + \frac{\sin(t_n\rho)}{\rho} \wh{u_1}(\rho)\Big) \sin(r\rho-\tau)(r\rho)^{-\nu-\frac12} \rho^{d}\, d\rho \\
&\qquad \int_0^\I \wh{w_n}(\sigma)\big( \tilde \omega_1(r\sigma) \cos(r\sigma-\tau) - \tilde\omega_2(r\sigma)\sin(r\sigma-\tau)\big)   (r\sigma)^{-\nu-\frac12}\sigma^d\, d\sigma\: e^{-\eps r} r^{d-1}\, dr
}
In view of~\eqref{IBP} the $r$-integral here is of the form
\EQ{\label{sdfg}
\int_t^\infty e^{-[\eps\mp i\tau]r} \, \omega(r\sigma)\, dr = \frac{e^{-[\eps\mp i\tau ]t}}{\eps \mp i\tau} \omega(t\sigma)
 + \int_t^\infty \frac{e^{-[\eps\mp i\tau ]r}}{\eps \mp i\tau} \omega'(r\sigma)\sigma \, dr
}
for all $t>0, \sigma>0$.  Inserting the boundary term on the right-hand side of~\eqref{sdfg} into~\eqref{U0w1} yields
expressions of the form, for $j=1,2$, 
\[
\int_0^\infty\int_0^\infty \frac{e^{\pm i t_n\rho} e^{-(\eps\pm i (\rho\pm \sigma))(t_n+A)} }{\eps\pm i (\rho\pm \sigma) } \tilde \omega_j(r_n \sigma) \wh{u_0}(\rho)\wh{w_n}(\sigma)
(\sigma\rho)^{\frac{d+1}{2}}\, d\rho d\sigma
\]
where the signs are chosen independently of each other.  By the $L^2$-boundedness of the Hilbert, respectively, Hankel transforms and the fact that $\sigma>\rho_*>0$, 
we conclude that uniformly in $\eps>0$ this expression is~$O(t_n^{-1})$.  Similarly, the integral on the right-hand side of~\eqref{sdfg} yields 
\[
\int_{r_n}^\infty \Big[ \int_0^\infty\int_0^\infty \frac{e^{\pm i t_n\rho} e^{-(\eps\pm i (\rho\pm \sigma))r} }{\eps\pm i (\rho\pm \sigma) } \tilde \omega_j'(r\sigma) \sigma\, \wh{u_0}(\rho)\wh{w_n}(\sigma)
(\sigma\rho)^{\frac{d+1}{2}}\, d\rho d\sigma\Big] \, dr
\]
Again by $L^2$-boundedness the expression in brackets is $O(r^{-2})$ uniformly in $\eps>0$ and $n$. Integrating this in~$r>r_n$ then yields $O(r_n^{-1})$ as before. 
This shows that the entire first term on the right-hand side of \eqref{U0w1} is $O(r_n^{-1})$. The second and third terms satisfy the same bound and we are done. 
\end{proof}

\subsection{Energy partition for profile decompositions}

We first recall the notion of a profile decomposition which originates in this form in~\cite{BG98}. It plays a fundamental
role in the  analysis of nonlinear equations at large energies. See for example~\cite{DKM11, DKM12} and~\cite{CKLS1, CKLS2}.  

\begin{defi}
We say that a sequence $(u_{0,n},u_{1,n}) \subset \dot H^1 \times L^2 $ admits a profile decomposition $(U_{\mr L}^j, \partial_t U_{\mr L}^j)_{j \in \m N} \subset \dot H^1 \times L^2$ (solutions to the linear wave equation \eqref{eq:lw}), with parameters $(\lambda_{j,n},t_{j,n})$, and remainder $w_n^J$ (also solutions to the linear wave equation \eqref{eq:lw}) if there holds
\begin{gather}
\begin{cases}
\ds u_{0,n} = \sum_{j=1}^J\frac{1}{\lambda_{j,n}^{d/2-1}} U_{\mr L}^j \left( - \frac{t_{j,n}}{\lambda_{j,n}}, \frac{x}{\lambda_{j,n}} \right) + w_{n}^J(0,x), \\
\ds u_{1,n} = \sum_{j=1}^J  \frac{1}{\lambda_{j,n}^{d/2}} \partial_t U_{\mr L}^j \left( - \frac{t_{j,n}}{\lambda_{j,n}} \frac{x}{\lambda_{j,n}} \right) + \partial_t w_{n}^J(0,x),
\end{cases} \label{profiles} \\
\text{where} \quad \lim_{J \to +\infty} \limsup_{n \to +\infty} \| w^J_n \|_{S(\m R^d)} =0, \nonumber
\end{gather}
and the parameters are pseudo-orthogonal, that is for all $i \ne j$,
\[ \left| \ln \frac{\lambda_{j,n}}{\lambda_{i,n}} \right| + \frac{|t_{j,n} - t_{i,n}|}{\lambda_{j,n}} \to +\infty \quad \text{as} \quad n \to + \infty. \]
\end{defi}

The following corollary is the first of the two main results of this section. 

\begin{cor}
\label{energy_expansion_cutoff}
Let $\{(u_{0,n},u_{1,n})\}$ be a bounded sequence in $\dot H^1 \times L^2$, and assume it admits a profile decomposition \eqref{profiles} with profiles $(\vec U_{\mr L}^j)_{j \in \m N}$, parameters $(\lambda_{j,n},t_{j,n})$, and remainder $w_n^J$. Let $t_n, r_n$ be two sequences. Then we have the Pythagorean expansion:
\begin{multline*} \int_{|x| \ge r_n} \big( | \nabla_x u_{0,n}(x)|^2 + | u_{1,n}(x) |^2 \big)\, dx \\
= \sum_{j=1}^J \int_{|x| \ge r_n} \frac{1}{\lambda_{j,n}^d} \left| \nabla_{x,t} U_{\mr L} \left( - \frac{t_{j,n}}{\lambda_{j,n}}, \frac{x}{\lambda_{j,n}} \right) \right|^2 \, dx + \int_{|x| \ge r_n} | \nabla_{x,t} w_{n}(0,x)|^2 \, dx + o_{n}(1).
\end{multline*}
\end{cor}

\begin{proof}
It suffices to prove that the cross terms go to $0$, i.e.
\begin{gather*}
\forall i \ne j, \quad \int_{|x| \ge r_n} \frac{1}{\lambda_{i,n}^{d/2}} \nabla_{x,t} U_{\mr L}^i \left( - \frac{t_{i,n}}{\lambda_{i,n}}, \frac{x}{\lambda_{i,n}} \right) \frac{1}{\lambda_{j,n}^{d/2}} \nabla_{x,t} U_{\mr L}^j \left( - \frac{t_{j,n}}{\lambda_{j,n}}, \frac{x}{\lambda_{j,n}} \right) \, dx \to 0, \\
\forall j \le J, \quad \int_{|x| \ge r_n} \frac{1}{\lambda_{j,n}^{d/2}} \nabla_{x,t} U_{\mr L}^j \left( - \frac{t_{j,n}}{\lambda_{j,n}}, \frac{x}{\lambda_{j,n}} \right) \nabla_{x,t} w_{n}^J(0,x) \, dx \to 0.
\end{gather*}
After scaling, this takes the expression
\begin{gather*}
\forall i \ne j, \quad \int_{|x| \ge r_n/\lambda_{i,n}}  \nabla_{x,t} U_{\mr L}^i \left( - \frac{t_{i,n}}{\lambda_{i,n}}, x \right) \frac{\lambda_{i,n}^{d/2}}{\lambda_{j,n}^{d/2}} \nabla_{x,t} U_{\mr L}^j \left( - \frac{t_{j,n}}{\lambda_{j,n}}, \frac{\lambda_{i,n}}{\lambda_{j,n}}  x \right) \, dx \to 0, \\
\forall j \le J, \quad \int_{|x| \ge r_n/\lambda_{j,n}} \nabla_{x,t} U_{\mr L}^j \left( - \frac{t_{j,n}}{\lambda_{j,n}}, x \right) \lambda_{j,n}^{d/2} \nabla_{x,t} w_{n}^J(0,\lambda_{j,n}^{d/2} x) \, dx \to 0.
\end{gather*}
In both cases we will use Lemma \ref{lem:bilinear}: we have to check the weak convergence. 
%Notice that
%\[ \frac{1}{\lambda_{j,n}^{d/2-1}} U_{\mr L}^j \left( \frac{t-t_{j,n}}{\lambda_{j,n}}, \frac{x}{\lambda_{j,n}} \right) = S(t- t_{j,n}) \left( \frac{1}{\lambda_{j,n}^{d/2-1}} U_{\mr L}^j \left(0, \frac{x}{\lambda_{j,n}} \right), \frac{1}{\lambda_{j,n}^{d/2}} \partial_t U_{\mr L}^j \left(0, \frac{x}{\lambda_{j,n}} \right) \right). \]
In the first case, we have
\begin{multline*}
 \nabla_{x,t} S\left( \frac{t_{i,n}}{\lambda_{i,n}} \right) \left( \frac{\lambda_{i,n}^{d/2-1}}{\lambda_{j,n}^{d/2-1}} U_{\mr L}^j \left( - \frac{t_{j,n}}{\lambda_{j,n}}, \frac{\lambda_{i,n}}{\lambda_{j,n}} x \right), \frac{\lambda_{i,n}^{d/2}}{\lambda_{j,n}^{d/2}} \partial_t U_{\mr L}^j \left( - \frac{t_{j,n}}{\lambda_{j,n}}, \frac{\lambda_{i,n}}{\lambda_{j,n}} x \right) \right) \\
= \frac{\lambda_{i,n}^{d/2}}{\lambda_{j,n}^{d/2}} \nabla_{x,t} U_{\mr L}^j \left( \frac{t_{i,n} - t_{j,n}}{\lambda_{j,n}}, \frac{\lambda_{i,n}}{\lambda_{j,n}} x \right).
\end{multline*}
>From pseudo-orthogonality, it is clear that this last expression tends weakly to 0 in $L^2$.
Let us focus on the second, then
\[
\nabla_{x,t} S\left( \frac{t_{j,n}}{\lambda_{j,n}} \right) (\lambda_{j,n}^{d/2-1} w_{n,0}^J(0,\lambda_{j,n} x), \lambda_{j,n}^{d/2} w_{n,1}^J(0,\lambda_{j,n} x)
= \lambda_{j,n}^{d/2} \nabla_{x,t} w_n \left( t_{j,n}, \lambda_{j,n} x \right).
\]
But by construction of a profile decomposition, for $j \le J$, recall that
\[ \lambda_{j,n}^{d/2} \nabla_{x,t} w \left( t_{j,n}, \lambda_{j,n} x \right) \weak 0 \quad \text{weakly in } L^2 \quad \text{as} \quad n \to +\infty. \qedhere \]
\end{proof}

\subsection{Asymptotic vanishing of Strichartz norms}

Our final goal  is to prove   the stability of the asymptotic vanishing
of global Strichartz norms for free radial waves under localization of the data, see Lemma~\ref{lem:tl2} below. In three dimensions, this was 
established in~\cite[Claim 2.11]{DKM11}. 
This statement will play an important role in the  applications of this paper to wave maps, see~\cite{CKLS1}, \cite{CKLS2}. 

\begin{lem}[{\cite[Lemma 4.1]{DKM11}}] \label{lem:scaled_lin_disp}
Let $v$ be a solution to the linear wave equation \eqref{eq:lw}, and $(t_n) \subset \m R$, $(\lambda_n) \subset \m R^*_+$ be two sequences. Define the sequence 
\[ v_n(t,x) = \frac{1}{\lambda_n^{d/2-1}} v \left( \frac{t}{\lambda_n}, \frac{x}{\lambda_n} \right). \]
 Assume that $\ds \frac{t_n}{\lambda_n} \to \ell \in \overline{\m R}$. Then
\begin{align*} 
& \text{If } \ell \in \{ \pm \infty \}, & \quad \limsup_{n \to \infty}  \| \nabla_{x,t} v_n(t_n) \|^2_{L^2(||x| - |t_n|| \ge R \lambda_n)} & \to 0 \quad \text{as} \quad R \to +\infty, \\
& \text{If } \ell \in \m R, & \quad 
\limsup_{n \to \infty} \| \nabla_{x,t} v_n(t_n) \|^2_{L^2(|\ln (x/\lambda_n)| \ge \ln R)} & \to 0 \quad \text{as} \quad R \to +\infty.
\end{align*}
\end{lem}

\begin{proof}
First consider the case $\ell \in \m R$, then notice that
\[ \| \nabla_{x,t} v_n(t_n) \|_{L^2( |\ln (|x|/\lambda_n)| \ge  \ln R)} = \| \nabla_{x,t} v(\ell) \|_{L^2(|\ln |x|| \ge \ln R)} + o_n(1), \]
from where the result follows. In the case $|\ell| = +\infty$, then
\[ \| \nabla_{x,t} v_n(t_n) \|_{L^2(||x| - |t_n|| \ge R \lambda_n)} = \| \nabla_{x,t} v(t_n/\lambda_n) \|_{L^2(||x| - |t_n/\lambda_n|| \ge R)}, \]
and the result follows from Theorem~\ref{thm6}.
\end{proof}

\begin{lem}[Claim A.1 in \cite{DKM11}] \label{lem:tl1}
Let $(u,\partial_t u)$ and $(w_n, \partial_t w_n)$ be solutions to the linear wave equation \eqref{eq:lw} bounded in $\dot H^1 \times L^2$, and let $(\lambda_n)_n$, $(\mu_n)$, $(t_n)_n$, $(s_n)_n$ be sequences of real numbers (with $\lambda_n, \mu_n >0$). Assume that
\begin{equation} \label{tl1:wn_wH}
\lambda_n^{d/2} \nabla_{x,t} w_n (t_n, \lambda_n x) \weak (0,0) \quad \text{weakly in} \quad L^2.
\end{equation}
If $\varphi$ is either a radial, compactly supported smooth function such that $\varphi=1$ (or $\varphi \equiv 1$) in a neighbourhood of $0$, we have
\begin{align}
\int \varphi \left( \frac{ x }{\mu_n} \right) \nabla_{x,t} w_n(s_n,x) \cdot \frac{1}{\lambda_n^{d/2}} \nabla_{x,t} u \left( \frac{s_n - t_n}{\lambda_n}, \frac{x}{\lambda_n} \right) \, dx & \to 0,  \label{tl1:conv1} \\
\text{and} \quad \int (1- \varphi )\left( \frac{x}{\mu_n} \right) \nabla_{x,t} w_n(s_n,x) \cdot \frac{1}{\lambda_n^{d/2}} \nabla_{x,t} u
\left( \frac{s_n - t_n}{\lambda_n}, \frac{x}{\lambda_n} \right) \, dx & \to 0. \label{tl1:conv2}
\end{align}
as $n\to\infty$. 
\end{lem}

\begin{proof}
By conservation of the linear energy for solutions to \eqref{eq:lw}, and we have
\begin{multline} 
\int \nabla_{x,t} w_n(s_n,x) \cdot \frac{1}{\lambda_n^{d/2}} \nabla_{x,t} u
\left( \frac{s_n - t_n}{\lambda_n}, \frac{x}{\lambda_n} \right)\, dx \\ = \int \lambda_n^{d/2} \nabla_{x,t} w_n(t_n,x) \cdot \nabla_{x,t} u
\left( \frac{s_n-t_n}{\lambda_n}, x \right)\,  dx  \\
 = \int \lambda_n^{d/2} \nabla_{x,t} w_n(t_n, \lambda_n x) \cdot \nabla_{x,t} u
\left( 0, x \right) \, dx \to 0.
\end{multline}
where we used weak convergence \eqref{tl1:wn_wH}. This settles the case $\varphi \equiv 1$. Also this shows that is suffices to prove \eqref{tl1:conv2}. For this we will use Lemma \ref{lem:bilinear}. Writing $\ds \varphi(z) = - \int \m 1_{[y \ge z]} \varphi'(y) dy$, we have
\begin{gather*}
\int \varphi \left( \frac{ x }{\mu_n} \right) \nabla_{x,t} w_n(s_n,x) \cdot \frac{1}{\lambda_n^{d/2}} \nabla_{x,t} u \left( \frac{s_n - t_n}{\lambda_n}, \frac{x}{\lambda_n} \right) \, dx = -\int_0^\infty 
\varphi'(y) F_n(y) dy, \\
\text{where} \quad F_n(y) = \int \m 1_{[x \le \mu_n y]} \nabla_{x,t} w_n(s_n,x) \cdot \frac{1}{\lambda_n^{d/2}} \nabla_{x,t} u \left( \frac{s_n - t_n}{\lambda_n}, \frac{x}{\lambda_n} \right) \, dx.
\end{gather*}
Unscaling, we see that
\[ F_n(y) = \int \m 1_{[x \le \mu_n y/\lambda_n]} \lambda_{n}^{d/2} \nabla_{x,t} w_n(s_n,\lambda_n x) \cdot  \nabla_{x,t} u \left( \frac{s_n - t_n}{\lambda_n}, x \right) \, dx. \]
Now we compute 
\EQ{
&\nabla_{x,t} S \left(  \frac{t_n-s_n}{\lambda_n} \right) \left(  \lambda_{n}^{d/2-1} w_n(s_n,\lambda_n x), \lambda_{n}^{d/2} \partial_t w_n(s_n,\lambda_n x) \right)\\
& = \lambda_{n}^{d/2}  \nabla_{x,t} w_n(t_n, \lambda x) \weak 0 \quad \text{in } L^2, 
}
by hypothesis. Hence \eqref{eq:Uw} ensures that for all $y$, $F_n(y) \to 0$ as $n \to +\infty$. Furthermore, it is clear that 
\[ |F_n(y)| \le \| (u,\partial_t u) \|_{\dot H^1 \times L^2} \| (w_n, \partial_t w_n) \|_{\dot H^1 \times L^2} \le M. \]
Hence for all $n$, $| \varphi'(y) F_n(y) | \le M |\varphi'(y)|$. As $\varphi' \in L^1$, the Theorem of dominated convergence applies and 
\[ \int_0^\infty  \varphi'(y) F_n(y) dy \to 0. \qedhere \]
\end{proof}

We are now in a position to derive the aforementioned stability result for the asymptotic vanishing of the Strichartz norms. 

\begin{lem}[Claim 2.11 in \cite{DKM11}] \label{lem:tl2}
Let $w_n$ be a sequence of radial solutions to the linear wave equation \eqref{eq:lw} with bounded energy and such that
\[  \| w_n \|_{S(\m R)} \to 0 \quad \text{as} \quad  n \to + \infty. \]
Let $(w_{0,n},w_{1,n})$ be the initial data of $w_n$, $\chi \in \q D(\m R^d)$ radial and such that $\chi = 1$ around the origin, and $\lambda_n$ be a sequence of positive numbers. Consider the solution $v_n$ to \eqref{eq:lw} with truncated data
\[ (v_{0,n}, v_{1,n}) := ( \varphi(|\cdot|/\lambda_n) w_{0,n}, \varphi(|\cdot| / \lambda_n) w_{1,n}), \]
where $\varphi =\chi$ or $\varphi = 1-\chi$. Then
\[ \| v_n \|_{S(\m R)} \to 0 \quad \text{as} \quad  n \to + \infty. \]
\end{lem}

\begin{proof}
It suffice to consider the case $\varphi = \chi$. By scaling invariance, we can assume that $\lambda_n =1$ for all $n$. 
Notice that convergence to 0 in the Strichartz space $S$ is equivalent to having trivial profile decomposition, more precisely, one has the following:

Let $(u_n,\partial_t u_n)$ be a sequence of solution to \eqref{eq:lw}. Then $\| u_n \|_{S(\m R)} \to 0$ if and only if for any sequence $(t_n) \subset \m R$, $(\mu_n) \subset (0,+\infty)$, 
\EQ{\label{iffclaim} \mu_n^{d/2} \nabla_{x,t} u_n (-\mu_n t_n, \mu_n x) \weak 0 \quad \text{weakly in } L^2. }

This a consequence of the construction of a profile decomposition, see \cite{BG98} for further details.

Hence, let $(t_n) \subset \m R$, $(\mu_n) \subset (0,+\infty)$ be two sequences, and $(u_0,u_1) \in \dot H^1 \times L^2$. By density, we can assume that $(u_0,u_1)$ are radial, smooth and compactly support outside 0, say in $\{ x | |x| \in [\rho_*,\rho^*] \}$ for some $\rho_*, \rho^* >0$. Define $(u,\partial_t u)$ be the solution to \eqref{eq:lw} with initial data $(u_0,u_1)$. It suffices to prove that
\[ \int \frac{1}{\mu_n^{d/2}} \nabla_{x,t} v_n \left( - \frac{t_n}{\mu_n}, \frac{x}{\mu_n} \right) \nabla_{x,t} u(0,x) \, dx \to 0. \]
Now we compute,
\begin{align*}
\MoveEqLeft \int \frac{1}{\mu_n^{d/2}} \nabla_{x,t} v_n \left( - \frac{t_n}{\mu_n}, \frac{x}{\mu_n} \right) \nabla_{x,t} u(x) \, dx = \int \nabla_{x,t} v_n \left( 0, x \right) \mu_n^{d/2} \nabla_{x,t} u(-t_n, \mu_n x) \, dx \\
& = \int \varphi \left( x \right)\nabla_{x} w_n \left( 0, x \right) \mu_n^{d/2} \nabla_{x} u(t_n,\mu_n x) \, dx \\&+ \int \varphi \left( x \right)  \partial_t w_n \left( 0, x \right) \mu_n^{d/2} \partial_t u(t_n,\mu_n x) \, dx \\
& \qquad + \int \nabla_x \varphi \left( x \right)  w_n \left( 0, x \right) \mu_n^{d/2} \nabla_{x} u(t_n,\mu_n x) \, dx 
\end{align*}
Then~\eqref{iffclaim} together with \eqref{tl1:conv1}, \eqref{tl1:conv2} shows that the first two terms of the right-hand side converge to 0. Hence we are left to prove that
\[ I_n := \int \nabla_x \varphi \left( x \right) w_n \left( 0, x \right) \mu_n^{d/2} \nabla_{x} u(t_n,\mu_n x) \, dx  \to 0. \]
It suffices to prove this for subsequences, hence we can assume that $t_n$, $\mu_n$ and $t_n/\mu_n$ have a limit in $\overline{\m R}$. The claim ensures that $w_n(0,x) \weak 0$ in $\dot H^1$. Recall that $\nabla_x \varphi$ has compact support away from 0: due to Hardy's inequality, we deduce
\[ \frac{1}{|x|} w_n(0,x) \weak 0 \quad \text{in } L^2\text{-weak}, \quad \text{and then} \quad \nabla_x \varphi \left( x \right) w_n \left( 0, x \right) \weak 0 \quad \text{in } L^2\text{-weak}. \]
In particular $\| w_n(0,x)/|x| \|_{L^2}, \| \nabla_x \varphi \left( x \right) w_n \left( 0, x \right) \|_{L^2}$ are bounded.

\medskip

First assume that $\ds t_n \to \tau \in \m R$. By Lemma \ref{lem:scaled_lin_disp}, we see that $\ds \mu_n^{d/2} \nabla_{x,t} u \left(t_n, \frac{x}{\mu_n} \right)$ concentrates $L^2$ mass on annuli of the form
\[ \{ x | \ \mu_n/R \le |x| \le \mu_n R \} . \]
Hence if $\mu_n \to +\infty$ or if $\mu_n \to 0$, as $\nabla_x \varphi$ has compact support away from 0, we see that $I_n \to 0$.

If $\mu_n \to \mu \in (0,+\infty)$, then 
\[ \mu_n^{d/2} \nabla_{x} u(t_n,\mu_n x) \to \mu^{d/2} \nabla_x u(\tau, \mu x) \quad \text{strongly in } L^2. \] 
As $\nabla_x \varphi \left( x \right) w_n \left( 0, x \right) \weak 0$ in $L^2$-weak, we deduce that $I_n \to 0$.

\medskip

We now turn to the case when $\ds |t_n| \to +\infty$. Then Lemma \ref{lem:scaled_lin_disp} shows that
\begin{equation} \label{u_Cn_3}
\limsup_{n \to +\infty} \| \mu_n^{d/2} \nabla_{x,t} u \left( t_n, \mu_n x \right) \|_{L^2( | |x| - |t_n|/\mu_n| \ge R/\mu_n)} \to 0 \quad \text{as} \quad R \to +\infty.
\end{equation}
If $\ds \frac{|t_n|}{\mu_n} \to + \infty$, then for all $R$, $\Supp (\nabla_x \varphi) \subset \{ | |x| - |t_n|/\mu_n| \ge \mu_n R \}$ when $n$ is large enough: hence we see that $I_n \to 0$.

Otherwise, $\ds \frac{t_n}{\mu_n} \to \ell \in \m R$. Using \eqref{u_Cn_3} and $\mu_n \to +\infty$, we see that
\[ \| \mu_n^{d/2} \nabla_{x,t} u \left(t_n, \frac{x}{\mu_n} \right) \|_{L^2( | |x| - \ell| \ge \mu_n^{-1/2})} \to 0. \]
Hence, separating the integral in $I_n$ between the regions $\{ |x| - \ell| \ge \mu_n^{-1/2} \}$ and its complement, and writing $\nabla_x \varphi \left( x \right) = |x| \nabla_x \varphi(x) \frac{1}{|x|}$, we have
\begin{align*}
I_n & = \int_{| |x| - \ell| \le \mu_n^{-1/2}} |x| \nabla_x \varphi \left( x \right) \frac{1}{|x|} w_n \left( 0, x \right) \mu_n^{d/2} \nabla_{x} u(t_n,\mu_n x) dx  + o(1) \\
& =  |\ell| \nabla_x \varphi \left( \ell \right)  \int_{| |x| - \ell| \le \mu_n^{-1/2}} \frac{1}{|x|} w_n \left( 0, x \right) \mu_n^{d/2} \nabla_{x} u(t_n,\mu_n x) dx  + o(1).
\end{align*}
(we used the continuity of $|x| \nabla_x \varphi \left( x \right)$ at $\ell$ on the last line, and $1/\sqrt{\mu_n} \to 0$). 

Now as $\ds \frac{1}{|x|} w_n \left( 0, x \right) \weak 0$ in $L^2$, we deduce that the last integral converges to 0, and $I_n \to 0$. This completes the proof.
\end{proof}

\centerline{\scshape Rapha\"el C\^{o}te }
\medskip
{\footnotesize
\begin{center}
CNRS and École polytechnique \\
Centre de Mathématiques Laurent Schwartz UMR 7640 \\
Route de Palaiseau, 91128 Palaiseau cedex, France \\
\email{cote@math.polytechnique.fr}
\end{center}
} 

\medskip

\centerline{\scshape Carlos Kenig, Wilhelm Schlag}
\medskip
{\footnotesize
% please put the address of the first author
 \centerline{Department of Mathematics, The University of Chicago}
\centerline{5734 South University Avenue, Chicago, IL 60615, U.S.A.}
\centerline{\email{cek@math.uchicago.edu, schlag@math.uchicago.edu}}
} % Do not forget to end the {\footnotesize by the sign }

\end{document}